\documentclass[10pt]{amsart}

\usepackage{latexsym,amsmath,amssymb,graphicx,epsfig,color,hyperref,supertabular,cite}
\usepackage{pdfpages}
\usepackage{pst-node}
\usepackage{tikz-cd} 
\usepackage[all,cmtip]{xy}
\usepackage[UKenglish]{babel}
\usepackage[T1]{fontenc}
\usepackage{amsmath,tikz}

\def\pplogo{\vbox{\kern-\headheight\kern -15pt
\halign{##&##\hfil\cr&{
\ppnumber}\cr\rule{0pt}{2.5ex}&\ppdate\cr} }} \makeatletter
\def\ps@firstpage{\ps@empty \def\@oddhead{\hss\pplogo}%
  \let\@evenhead\@oddhead 
  
}
\renewcommand{\thetable}{\@Alph\c@table}
\makeatother

\DeclareFontFamily{OT1}{rsfs10}{}
\DeclareFontShape{OT1}{rsfs10}{m}{n}{ <-> rsfs10 }{}
\DeclareMathAlphabet{\mathscript}{OT1}{rsfs10}{m}{n}

\suppressfloats
\def\ppnumber{\vbox{\baselineskip16pt }}
\def\ppdate{}
\date{} 

\newsavebox{\fmbox}

\author{Antonella Grassi and Timo Weigand}

\title[On topological invariants of threefolds with singularities]{On topological invariants of algebraic threefolds with ($\Q$-factorial) singularities}

  \address{Antonella Grassi, Department of Mathematics, University of Pennsylvania, Philadelphia, PA 19104}
  \email{grassi@math.upenn.edu}
  \address{Timo Weigand, Theory Department, CERN, CH-1211 Geneva, Switzerland and Institut f{\"u}r Theoretische Physik, Ruprecht-Karls-Universit\"at Heidelberg, Philosophenweg 19, 69120 Heidelberg, Germany}

\theoremstyle{plain}
\newtheorem{thm}{Theorem}

\newtheorem{corollary}[thm]{Corollary}

\newtheorem{proposition}[thm]{Proposition}
\newtheorem{lemma}[thm]{Lemma}

\newtheorem{method}[thm]{Assignment}
\newtheorem{conjecture}[thm]{Conjecture}
\theoremstyle{definition}
\newtheorem{definition}[thm]{Definition}
\newtheorem{ex}[thm]{Example}

\newtheorem{definitionp}[thm]{Definition-Proposition}

\theoremstyle{remark}
\newtheorem{remark}[thm]{Remark}

\numberwithin{thm}{section}

\newcommand{\cd}{\mathrm{CxDef}}

\newcommand{\cO}{{\mathcal O}}
\newcommand{\cU}{{\mathcal U}}

\newcommand{\cl}{{\mathcal Cl}}

\def\C{\mathbb{C}}
\def\Z{\mathbb{Z}}

\def\Q{\mathbb{Q}}

\def\P{\mathbb{P}}

\def\r{\boldmath  \mathcal R}

\newcommand{\g}{{\mathfrak g}}
\newcommand{\T}{{\mathfrak a}}
\newcommand{\D}{{\mathfrak d}}
\newcommand{\E}{{\mathfrak e}}

\def\Pic{\operatorname{Pic}}
\def\Spec{\operatorname{Spec}}

\def\WDiv{\operatorname{WDiv}}
\def\Div{\operatorname{CDiv}}

\def\rk{\operatorname{rk}}

\def\codim{\operatorname{codim}}

\def\cD{\operatorname{CxDef}}
\def\kD{\operatorname{KaDef}}
\def\ns{\operatorname{NS}}
\def\NE{\operatorname{NE}}

\newcommand{\be}{\begin{equation}}
\newcommand{\ee}{\end{equation}}
\newcommand{\bea}{\begin{eqnarray}}
\newcommand{\eea}{\end{eqnarray}}
\def\su{\mathfrak{su}}
\def\spin{\mathfrak{so}}
\def\sp{\mathfrak{sp}}

\newcommand{\adj}{\operatorname{adj}}
\newcommand{\fund}{\operatorname{fund}}

\newcommand{\spinrep}{\operatorname{spin}}
\newcommand{\vect}{\operatorname{vect}}

\def\so{{\mathfrak {so}}}
\newcommand{\nonmin}{NM}
\newcommand{\terminal}{NSR}

\begin{document}
\date{\today}

\ppnumber{\qquad \quad \quad \quad \quad \qquad \quad \quad \quad \quad\qquad \quad \quad \quad \quad  \qquad \quad \quad \quad \quad \qquad \quad \quad \quad \quad \quad CERN-TH-2018-013}

\begin{abstract} 
We study local, global and local-to-global properties of threefolds with certain singularities.
We prove criteria for these  threefolds to be rational homology manifolds and conditions for threefolds to satisfy  rational Poincar\'e duality. We   
relate the  topological Euler characteristic  of   elliptic Calabi-Yau threefolds with $\Q$-factorial terminal singularities to dimensions of  Lie algebras and certain representations, Milnor and Tyurina numbers   and  other birational invariants of an
elliptic fibration.  We give an interpretation in terms of complex deformations.
We state a  conjecture on the extension  of   Kodaira's classification of  singular fibers on relatively minimal elliptic  surfaces  to the class of  birationally equivalent  relatively minimal  genus one fibered varieties and we give results in this direction.

\end{abstract}

\maketitle
\section{Introduction}
We study local, global and local-to-global properties of threefolds with singularities, in particular  terminal and   klt, $\Q$-factorial singularities. The computation of the topological  Euler characteristic 
of a genus one fibered variety is a known  illustration of such properties:  the non trivial contributions  are localized in the stratified singular loci of the fibration and are combined  via the Mayer-Vietoris Theorem.
Poincar\'e duality is another.   For complex  varieties  local  and global deformations  are of interest, and  so are local-to-global principles relating local and global deformations.
 A less known occurrence is the  expected correspondence, predicted by physics, between 
elliptic fibrations of smooth Calabi-Yau varieties and Lie algebras  together with their representations.

We state and formalize 
 the   correspondence between fibrations and algebras in terms of    local and  global properties of  the stratified singular locus of elliptic fibrations (not necessarily smooth). We  then prove related  local-to-global properties for Calabi-Yau threefolds with $\Q$-factorial terminal singularities.
In particular we   prove a formula which relates the dimensions of Lie algebras and certain representations to the topological Euler characteristic  of   elliptic Calabi-Yau threefolds with $\Q$-factorial terminal singularities, Milnor and Tyurina numbers   and  other birational invariants of the  (relative minimal) elliptic fibration (Definition \ref{defr} and Theorem \ref{Anomalycondition1}). We 
 state precise conjectures for this correspondence for more general varieties and fibrations.
Our results constitute   a 
 step towards what we call a ``Grothendieck-Brieskorn" program regarding the intriguing connection between singular varieties and Lie algebras, together with certain distinguished representations.  
In fact, Brieskorn, Grothendieck,  and later Slodowy \cite{BrieskornGrot,Slodowy}, discovered beautiful connections between surface singularities and Lie algebras, but a mathematical explanation  of the parallelisms of the classification  between singularities and Lie algebras remains elusive (Arnol'd,
\cite{Arnold}).
A well known illustration of the parallelism is the  correspondence between surface rational double points  and Lie algebras: 
 rational double points
are classified by the Dynkin diagrams of the
simply laced Lie algebras of type $\T, \D,\E$  \cite{DV,coxeter:annals}.
These are also the singularities of Weierstrass elliptic surfaces.

Although Calabi-Yau threefolds with $\Q$-factorial terminal singularities  and elliptic fibrations are the focus of the applications in Section \ref{chitop},  we state our program in more generality,  in particular for higher dimensions and different singularities.   In Section \ref{BG}  we review and extend the correspondence between fibrations and Lie algebras and their representations, and state  local-to-global principles which relate them.
The correspondence is expressed in terms  of  the components of codimension one and two of the stratified singular loci of the fibrations. 
  (A different mathematical approach  using deformation is taken in  ongoing projects \cite{GrassiHalversonShanesonJS}.)
   Key ingredients for the  proofs in Sections \ref{BG} and \ref{chitop} are several other local, global and local-to-global results which we prove in the preceding sections.

 In  Section \ref{ClGQFRP} we start with a review of different notions of cohomologies
  which agree in the case of rational homology manifolds; rational homology manifolds satisfy Poincar\' e duality. 
We then establish necessary and sufficient conditions for threefolds with certain  rational and klt  singularities to be rational homology manifolds (Theorem \ref{kltRHMT} with J. Shaneson, and Theorem \ref{klt}) and to satisfy rational Poincar\'e duality (Theorems \ref{rPD1} and   \ref{rPD2}). These conditions are expressed in terms of the properties of  local analytic and (global) algebraic $\Q$-factoriality as well as local-to-global principles.
To this end, we prove a local analytic $\Q$-factorialization result for isolated klt 
  singularities (Theorem \ref{Klqf}).
The above  singularities occur naturally in the minimal model program as well as in  various physics models. The minimal varieties in the sense of the Mori Minimal Model Programs have generally terminal $\Q$-factorial singularities. In string theory 
it is known that some Calabi-Yau fourfolds with terminal singularities are the ``correct" models, even when a smooth  birationally equivalent minimal  Calabi-Yau exists  \cite{DDFGK}.
Even for Calabi-Yau threefolds 
 the appearance of $\Q$-factorial terminal singularities seems oftentimes unavoidable \cite{BraunMorrison,LustTFects,Morrison:2016lix,Font:2017cya}. While klt singularities of varieties (not of pairs $(B, \Delta)$) have not yet appeared  in the physics literature that we know of, we speculate that they  might  occur naturally as boundary components, in the study of the heterotic/F-theory duality.   Generalized  ``Calabi-Yau" threefolds with isolated klt but not canonical singularities are of interest in mathematics, for example regarding their structure and rationality properties (Remark \ref{R3}).

In Section \ref{motivation} we recall some known properties of smooth Calabi-Yau threefolds, in particular elliptically fibered ones, and state the original motivation of our work.  We review the geometric and algebraic definitions of global and local analytic and algebraic $\Q$-factoriality, the local Picard group, as well as of the algebraic and analytic Class groups in Section \ref{basics}.
In Section  \ref{FQFAF} we  also summarize some properties of factoriality and $\Q$-factoriality which are hard to find in the literature. In Section \ref{CYCond} we discuss the ``Calabi-Yau" condition, in any dimension, and $\Q$-factoriality.

 In Section  \ref{cxDefMT} we turn to Calabi-Yau threefolds with $\Q$-factorial terminal singularities.  We find that  the  dimension of the complex deformation space is computed by $b_3(X)$ with a modification coming from the dimension of the versal deformations of each singularity, the Tyurina/Milnor number in this case.  The  decomposition  (\ref{CdefY3a}) suggests the existence of a ``local-to-global principle" (Conjecture \ref{localtoglobaldef}). Section \ref{b2} concerns K\"ahler deformations.  The results about complex and K\"ahler deformations  play a role in the proof of the results in Section \ref{chitop}. Some of the results in Section \ref{chitop} generalize previous results of \cite{GrassiMorrison03} for smooth elliptic Calabi-Yau threefolds. 
 
In Section \ref{FtheoryInterpretation} we  review the basics of the predicted correspondence by the physics of F-theory and interpret our results.  Applications to physics are studied in companion papers \cite{ArrasGrassiWeigand} and \cite{GrassiWeigand2}.

We state a precise conjecture (Conjecture \ref{KodairaExt}) on the extension  of   Kodaira's classification of  singular fibers on relatively minimal genus one surfaces to the class of  birationally equivalent  relatively minimal  genus one fibered varieties.  We support  the conjecture with  a local-to-global principle, by associating to the stratified discriminant  locus   of the fibration $\Sigma$ the non abelian  and abelian gauge algebras and their representations.
We give results in this direction.

\vskip 0.1in

Because this paper spans from  algebraic geometry, algebra and topology, with applications to string theory,  we have also  included some general definitions and properties.

\vspace{5pt}

\noindent{\bf  Acknowledgements}

We thank P. Aluffi, M. Banagl,  J. Koll\'ar, L. Maxim, R. Pardini, M. Rossi,   R. Svaldi, S. Verra and especially J. Shaneson, for comments on the draft and useful discussions. We also thank the Aspen Center for Physics, the Fields Institute in Toronto, and the Banff International Research Station
for hospitality during various stages of this project. 
The work of TW is partially supported by DFG under TR33 ``The Dark Universe".

\section{Motivation: smooth Calabi-Yau varieties }\label{motivation}
  \begin{definition} Let $X$ be a complex normal algebraic threefold. $X$ is  Calabi-Yau if $h^i(X, \cO_X)=0 $, $  i  = 1,2, $ and $K_X \sim \cO_X$.
\end{definition}
If  $X$ is also  smooth, Poincar\'{e} duality and the Hodge decomposition imply that  the Betti numbers (of the singular cohomology) are  
\begin{equation}\label{CY}
b_1=0, \ \  b_2= h^{1,1}(X), \ \ b_3= 2(1+h^{2,1}).
\end{equation}
The topological Euler characteristic can be written as
\begin{equation}\label{chi}
\chi_{top}(X)=2 \{h^{1,1}(X)-h^{2,1}(X)\}.
\end{equation}
The above expression is particularly relevant in physics applications  because the  Picard group  $\Pic(X)$ is isomorphic to  $H^2(X, \Z)$;  its rank, which is also called  the rank of the N\'eron-Severi group $\ns(X)$, counts the K{\"a}hler  deformations $h^{1,1}(X)$ of $X$.  Similarly $ h^{2,1}(X)$ is the dimension of the  space of complex deformations of $X$, the Kuranishi space  of $X$.
 When $X$ has singularities, these identifications do not necessarily hold, so we denote by $\kD$ the dimension of the space of K{\"a}hler deformations and  by $\cD$ the dimension of the space of complex deformations.
 In the smooth case  \eqref{chi} becomes
\begin{equation}\label{Dchi}
\chi_{top}(X)=2 \{\kD(X)-\cD(X)\}.
\end{equation}
The celebrated mirror symmetry implies that  mirror pairs of  smooth Calabi-Yau  threefolds have topological Euler characteristics of opposite sign.

Equation \eqref{Dchi}  is used in \cite{GrassiMorrison03}  to express the topological Euler  characteristic of  general elliptically fibered smooth  Calabi-Yau threefolds (with a section) in terms of the Lie algebras and their representations which are  naturally associated to the singular fibers of the fibration. The results of \cite{GrassiMorrison03,GrassiMorrison11}  are motivated by the ``anomaly cancellation" mechanism in physics. In fact, the cancellation of anomalies interpreted in the geometry of smooth elliptically fibered Calabi-Yau threefolds \cite{GrassiMorrison11} involves $\kD(X)$ and $\cD(X)$ as defined above. In Section \ref{chitop} we generalize these results to singular Calabi-Yau threefolds. We focus on  the singularities which occur naturally in the minimal model program, i.e.~terminal, and klt, $\Q$-factorial singularities. We recall the basics in the following Section \ref{basics}.
 
\section{Terminal, canonical, klt, ($\Q$-)factoriality}\label{basics}
 Let $X$ be a complex normal reduced algebraic variety.  
\subsection{Terminal, canonical, klt}
A resolution 
of $X$ is a birational (bimeromorphic) morphism $
\rho: Y \rightarrow X$, with  $Y$ smooth.
  $X$ has rational singularities if and only if  $R^i\rho_* \cO_Y =0,\  i >0$.   $X$ has $\ell$-rational singularities if  and only if $R^i \rho_* \cO_Y =0,\  0 < i  \leq \ell$. 
  Let $j: X_0  \hookrightarrow X$ be the natural inclusion of the smooth locus, with canonical bundle $\omega_0(X_0)$  and let ${\omega_X}= j_*(\omega_0(X_0))$.  Since $X$ is normal, $\omega_X$ is  a reflexive sheaf;  let $K_X$ be the associated Weil divisor.

    \begin{definition}$X$ is {\it $\Q$-Gorenstein} if there exists some integer $r$ such that $rK_X$ is a Cartier divisor (that is $K_X$ is $\Q$-Cartier). The  minimum such integer $r$ is the index of $X$.  

  $X$ is Gorenstein if it is Cohen-Macaulay and of index $1$.
   \end{definition} 
  Let $X$ be a  $\Q$-Gorenstein variety  and  $\rho: Y \to X$ a smooth resolution with  
exceptional divisors $E_i$. If  a resolution $\rho$ is an isomorphism in codimension one, the resolution is called small.
In general
$
rK_{Y} = \rho^*rK_X + \sum_i a_i  r E_i \,, 
$
where  $a_i \in \Q$ are the {\it discrepancies}. 
\begin{definition} 
\begin{enumerate}
\item If $a_i > 0$ for all $i$, $X$ is said to have {\it at worst terminal} singularities. 
\item  If  $a_i \geq 0 $, for all $i$, $X$ is said to  have   {\it{at worst canonical }} singularities.
\item  If $a_i > -1$ for all $i$, $X$ is said to have {\it at worst  klt} singularities. 
\item 
 If $a_i  \geq  -1$ for all $i$, $X$ is said to have {\it at worst  log canonical} singularities. 
\end{enumerate}
\end{definition}

A smooth variety has at worst terminal singularities. If $X$ is a surface it can be shown that  $X$ has at worst terminal singularities if and only if it  is smooth. The canonical  surface singularities are the $\mathfrak a,  \mathfrak d,  \mathfrak e $ singularities (rational double points). 
  A singularity is klt if and only if it is a cyclic quotient of an index $1$ canonical singularity by an action which is fixed point free in codimension $2$ \cite[Corollary 5.21]{KollarMori}.

 \begin{thm}[See for example  6.2.12, \cite{IshiiLibro}]\label{kltRat}  Klt  singularities, hence canonical and terminal singularities, are rational  and in particular Cohen-Macaulay.
 \end{thm}


\begin{remark}\label{cdvHyp}    Recall that threefold terminal singularities  are isolated; if they are Gorenstein, they are analytic hypersurface singularities, see for example \cite{KollarMori}.  $X$ has at worst canonical singularities and the  index  of  $K_X$ is $1$ if and only if it is {Gorenstein} and rational.
 In particular a  Calabi-Yau variety with canonical singularities is Gorenstein. 

   \end{remark}
  
  Log canonical singularities are  however not necessarily rational; thus we will only consider here varieties with at worst klt singularities.
  
\subsection{Factorial, $\Q$-factorial; the geometric  definitions}\label{tckf}  ~\vskip 0.01in

\noindent The different  notions of  $\Q$-factori\-ality - algebraic, analytic, global and local - are quite delicate. We can consider $X$ as a complex algebraic variety, and also its support  $X^h$ as a complex analytic variety. Sometimes we will consider  complex analytic varieties $X$.  In the following we do not  always distinguish between $X$ and $X^h$, rather we specify if  the relevant objects are algebraic and analytic when necessary. Note that $X$ has rational singularities if and only if $X^h$ has rational singularities. 
\begin{definition}
Let $\WDiv(X)$ be the group of algebraic (analytic) Weil divisors and $\Div(X) $ the group of algebraic (analytic)  Cartier divisors.
\end{definition}

 If $X$ is smooth, then $\WDiv(X)=\Div(X)$. 
 This is true more generally if $X$ is factorial, that is if and only if, $\forall \ x \in X$, the local rings $\cO_{X,x}$ 
  are unique factorization domains.

 \begin{definition}\label{Qglobal}
Let $\WDiv(X)_{\Q}=\WDiv(X) \otimes_{\Z} \Q$  and $\Div_{\Q}(X)= \Div(X) \otimes_{\Z} \Q.$  \\
If $\WDiv(X)_{\Q}=\Div(X)_{\Q}$, $X$ is algebraic $\Q$-{\it factorial}  ({globally analytically}  $\Q$-{\it factorial}), or it has  (analytic)  $\Q$-{\it factorial} singularities.
\end{definition}

Often one says $\Q$-factorial instead of algebraic $\Q$-factorial.  
In the case of hypersurfaces and complete intersection varieties,   various results, starting from Grothendieck's work, relate  factoriality and $\Q$-factoriality,  for example \cite{PolizziRapagnettaSabatino}  and \cite{LinUFD}.
Analytic factoriality implies factoriality, but the converse is not true   \cite[Corollary 2,  p. 41]{Samuel}.
 In particular: 

  \begin{proposition}[Reid-Ue; Corollary 5.1  \cite{Kawamata1988}; \cite{LinUFD}] Let $(\mathcal U,p)$ be a threefold with terminal   singularities of index 1.  Then any $\Q$-Cartier divisor is Cartier, that is $\Q$-factoriality is equivalent to  factoriality.
\end{proposition}
Thus a Calabi-Yau threefold with $\Q$-factorial terminal singularities is factorial (Remark \ref{cdvHyp}).
 We also consider properties of  the germ  $(\mathcal U, p)$:
\begin{definition}
 $(\mathcal U, p)$ is    algebraic (analytically) $\mathbb Q$-factorial if every algebraic (analytic) Weil divisor in a neighborhood of $p$ (Euclidean in the analytic case) is $\mathbb Q$-Cartier. 
 \end{definition}
 
 \begin{definition}\label{laqf} $X$ is locally (analytically) $\Q$-factorial if  $(\mathcal U, p)$ is    algebraic (analytically) $\mathbb Q$-factorial, for every $ p \in X$ and any open set $\mathcal U$.
 \end{definition}

In the analytic case global does not  imply local necessarily, since there can be Weil divisors in  $(\mathcal U, p)$ which do not extend to Weil divisors in $X$. 
In general,  algebraic $\Q$-factoriality does not imply analytic $\Q$-factoriality; it does so if $X$ is projective algebraic, by Chow's Theorem. 
  \begin{ex}[\cite{Kawamata1988}\label{exfacnotQfac}, page 104, and \cite{IshiiLibro}]  $(\mathcal U, 0)=(xy+ zw+ z^3+w^3=0, 0)$ is $\Q$-factorial, actually factorial, but not locally analytically $\Q$-factorial, since the local analytic equation in $(\mathcal U, 0)$ is $z_1z_2+ z_3z_4=0$. 
\end{ex}
 When the isolated singularity is {toric}, the singularity is $\Q$-factorial if and only if the maximal cone corresponding to the toric singular point is { simplicial}. More generally a variety with orbifold singularities is $\Q$-factorial.

Following Kawamata we make
\begin{definition}
\begin{enumerate}
\item $ \sigma(X)=\dim (\WDiv(X)_{\Q} / \Div(X)_{\Q})$.
\item $ \sigma((\mathcal U, p))=\dim (\WDiv((\mathcal U, p))_{\Q} / \Div((\mathcal U,  p))_{\Q})$.
\end{enumerate}
\end{definition}

\begin{remark}\label{sigmaandQf} $X$ (respectively, $(\mathcal U, p)$) is   $\Q$-factorial if and only  if   $\sigma (X)=0$ (respectively, $\sigma (\mathcal U, p)=0$), that is  if and only if $\WDiv(X)/ \Div(X)$ (respectively, $\WDiv(\cU, p))/ \Div(\cU, p)$) is torsion. 
 \end{remark}

In Example  \ref{exfacnotQfac} $ \sigma((\mathcal U^h, 0))=1$.

 \begin{proposition}[Kawamata \cite{Kawamata1988}, Lemma 1.2]  Let $X$ be an algebraic variety.
 If $X$ has rational singularities, then 
$ \sigma(X)$
 is finite.   \end{proposition}

 \begin{proposition}[Kawamata \cite{Kawamata1988}, Lemma 1.12]\label{ratFin}  Let  $(\mathcal U, p)$ analytic.  If  $(\mathcal U, p)$  has rational singularities, then $
  \sigma((\mathcal U,p))$
  is finite.   \end{proposition}
 
 \smallskip
In particular we  use Property \ref{ratFin} in Theorem \ref{Klqf}.
In  Section \ref{FQFAF} we study $\Q$-factoriality and local analytic $\Q$-factoriality by analyzing  properties of the local class group and the local Picard group.

  \begin{definition}[The divisor class group, the  local divisor class group]  \hfill \\
  The Divisor Class group ${\mathcal Cl}(X)$ is the group of Weil divisors modulo the principal divisors. \\
The Picard group $\Pic(X)$ is the group of Cartier divisors modulo the principal divisors.  \\
The local (analytic) divisor class group,  ${\mathcal Cl}(\cO_{X,x})$ and  ${\mathcal Cl}(\cO^h_{X,x})$, is  geometrically the group of  local Weil divisors modulo the principal divisors.
  \end{definition}
  The following  results \ref{hartshorne1}, \ref{hartshorne2}  hold both in the algebraic and analytic settings \cite{Hartshorne}:
  
\begin{proposition}\label{hartshorne1}  ${\mathcal Cl}(\cO_{X,x})= \lim \limits_{\rightarrow} \Pic (\mathcal V)$, where the limit is taken over the set of $\mathcal V \subset \Spec (\cO_{X,x})$ such that ${\cO_z}$ is factorial $\forall \  z  \in  {\mathcal V}$.

 ${\mathcal Cl}(\cO^h_{X,x})= \lim \limits_{\rightarrow} \Pic (\mathcal V)$, where the limit is taken over the set of $\mathcal V \subset \Spec (\cO^h_{X,x})$ such that ${\cO_z}$ is factorial $\forall \  z  \in  {\mathcal V}$.
\end{proposition}

 \begin{corollary}\label{hartshorne2} Let  $\mathcal U$ be an  open set and  $Z \subset \mathcal U$ proper, closed, with $\codim Z \geq 2$. Then
 \bea {\cl} ({\mathcal U}) \simeq \cl ({\mathcal U} \setminus Z).
 \eea
 \end{corollary}
If $Z$ is the singular locus of $X$, then $Pic(\mathcal U \setminus Z) \simeq  {\cl} ({\mathcal U}) $.
If $\dim \mathcal U \geq 3$ and $ p \in \mathcal U$ an isolated singularity,  $\Pic (\mathcal U \setminus p)  \simeq {\cl} (\cO_{\mathcal U,p})$.
\begin{definition}\label{localPic} $\Pic (\mathcal U \setminus p)$ is usually called the local Picard group,  but the notation in the literature is not uniform; we refer to \cite{KollarLocPic}. 
In the following Section \ref{FQFAF} we review equivalent definitions and prove some properties.
\end{definition}

\section{ Almost factoriality and other properties; $\Q$-factorializations }\label{FQFAF}
\subsection{The algebraic definition}
 Let $A$ be a local noetherian normal domain.
\begin{definition}[The local divisor class group]  ${\mathcal Cl}(A)$ is the quotient of the divisorial ideals $\WDiv(A)$ modulo the principal ideals in $A$; ${\mathcal \Pic }(A)$ is the quotient of the invertible ideals $\Div(A)$ modulo the principal ideals in $A$. \end{definition}

The divisorial ideals are the  Weil divisors in $\mathcal U= {\rm Spec} (A)$.  We will apply the  definitions and results stated below  to the local ring $\cO_{X,x}=A$ and ${\rm Spec}(\cO_{X,x})=\mathcal U$.
Recall that $A$ is factorial if and only if it is a unique factorization domain.
In the algebra literature, see for example \cite{Fossum}, $\Q$-factorial is referred to as almost factorial:

  \begin{definition} $A$ is almost factorial (respectively factorial) if  ${\mathcal Cl }( A)$ is torsion (respectively zero).
  \end{definition}
  
\begin{proposition}\label{BerltramettiOdettiFossumStorch} The local class group   ${\mathcal Cl}(A)$ is torsion (zero)  if and only if \\
$ \WDiv(A) / \Div(A)$  is torsion (respectively $\WDiv(A) / \Div(A)=0$).
\end{proposition}
\begin{proof}  It follows from   \cite{BeltramettiOdetti, Fossum, Storch}.
\end{proof}

Equivalently, $A$ is almost factorial (respectively factorial) if and only if $ \WDiv(A)_{\Q} / \Div(A)_{\Q}$  is torsion  (respectively $\WDiv(A) / \Div(A)=0$).

\subsection{Geometry}
In the case of  isolated singularities one can  directly prove
 \begin{proposition}\label{localanclassgroup} Let  $(\mathcal U, p)$ be an analytic contractible germ 
  open set and $ p \in \mathcal U$ an isolated singularity. Then
 \bea {\cl} (\cO^h_{\mathcal U, p}) \simeq \Pic (\mathcal U \setminus p) \simeq \WDiv(\mathcal (\mathcal U, p)) / \Div(\mathcal (\mathcal U, p)) \,.
 \eea
 \end{proposition}
  Therefore, by Remark \ref{sigmaandQf} $(\mathcal U, p)$ is  analytically $\Q$-factorial if and only  if   $\sigma ({\mathcal U},p)=0$, that is  the local  analytic divisor class group $ {\cl} (\cO^h_{{\mathcal U, p}})$ is torsion.

  \subsection{$\Q$-factorializations}
  
Recall that the  nodal quintic threefold  $ X \subset \P^4$ of equation $  x_0g_0 + x_1 g_1=0$, with $g_0, \ g_1$ general quartic polynomials in the variables $[x_0, \cdots, x_4]$, is not $\Q$-factorial, as the  Weil divisor  $D$  defined by $x_0=g_1=0$  is not Cartier. Recall that the  birational morphism obtained by blowing up $\P^4$ along $D$ provides a small projective  resolution $X_1 \to X$; in particular $X_1$ is a  $\Q$-factorial variety.    When the isolated singularity is {toric}, $X$ is not $\Q$-factorial if and only if the maximal cone corresponding to the toric singular point is not { simplicial}; a $\Q$-factorial birational model $X_1$ together with a small morphism $X_1 \to X$  is achieved by  a simplicial subdivision of the cone. More generally, we have

 \begin{thm}[Corollary 1.4.3, \cite{BCHMcK}
 ]\label{Aqf}   Let $X$ be an algebraic threefold with klt singularities.
  If $X$ is not $\Q$-factorial, there exists a small projective birational morphism $\phi: X_1 \to X$, where $X_1$ is $\Q$-factorial with klt singularities.
 \end{thm}
 
The above algebraic $\Q$-factorialization Theorem \ref{Aqf} was first proved by Kawamata in Corollary 4.5 \cite{Kawamata1988},  for a threefold analytic germ with at most terminal singularities. In the proof Kawamata uses the classification of threefold terminal singularities and he reduces to singularities of index one.  
Birkar, Cascini, Hacon and M\textsuperscript{c}Kernan prove Corollary 1.4.3 in  \cite{BCHMcK}  as a consequence of their celebrated Theorem of the finite generation of the canonical ring.
   
  Kawamata in Corollary 4.5' \cite{Kawamata1988}  also proves
 the  analytic $\Q$-factorialization for  a threefold analytic germ with terminal singularities.  We need  an analytic $\Q$-factorialization  result for klt singularities.

 \begin{thm}\label{Klqf}   Let $(\mathcal U^h, p)$ be a threefold analytic germ with at most an isolated klt singularity at $p$.
  If $(\mathcal U^h, p)$ is not  analytically $\Q$-factorial, there exists a small  bimeromorphic morphism $\phi:\mathcal V^h \to \mathcal U^h$, where $\mathcal V^h$ is analytically $\Q$-factorial with klt singularities.
 \end{thm}

  We first need\footnote{We thank J. Koll\'ar for pointing out Artin's \cite{ArtinApprox}.}
 \begin{thm}[Algebraic Approximation, \cite{ArtinApprox}, Th. 3.8]\label{AlgApprx} Let  $(\mathcal U^h, p)$ be a threefold analytic germ with an isolated singularity at $p$ and $D^h$ a Weil divisor. There is a normal quasi-projective variety $ X$, an open neighborhood  $\mathcal U  \subset  X $, $p \in \mathcal U$,  and a Weil divisor $D$ such that, possibly after restricting $\mathcal U^h$, there exists a biholomorphic map $m: \mathcal U^h \to \mathcal U$ with $m(D^h)=D$.
  \end{thm}
In Artin's result the singularities are isolated.
 \begin{proof}[Proof of Theorem \ref{Klqf}] Let $D$ be a generator of $\mathcal Cl ({\cO^h_{\mathcal U^h, p}})$. After possibly further restricting to smaller open neighborhoods of $p$, we have that $D$ is a generator of $\mathcal Cl ({\cO_{\mathcal U, p}})$ and $\WDiv(\mathcal U) / \Div(\mathcal U)$ (Theorem \ref{AlgApprx}). By Theorem \ref{Aqf} there exists a small projective birational morphism $\phi: {\mathcal U}_1 \to  {\mathcal U}$, where $ {\mathcal U_1}$ is $\Q$-factorial with klt singularities. Then $m^{-1} \cdot \phi: {\mathcal U}^h_1  \to \mathcal U^h$ is a small bimeromorphic morphism, and $\sigma(\mathcal U^h_1)  < \sigma (\mathcal U^h)$. In addition, $\mathcal U^h_1$ has klt singularities. Since $\sigma (\mathcal U^h)$ is finite (Theorems \ref{kltRat} and \ref{ratFin}) by repeating the process if necessary we obtain the Theorem.
 \end{proof}

Also Koll\'ar kindly provided us another  proof:  Theorem 3.10  in \cite{ArtinApprox} and  \cite{HironakaFormal} implies that for any finitely generated subgroup   $G\subset \mathcal Cl(\cO^h_{\mathcal U^h, p})$ there is an algebraic approximation $\mathcal U$ such that $G$ is contained in the image of $\mathcal Cl(\cO_{\mathcal U, p})$.

In general the  relation between the algebraic and analytic class groups is quite delicate.
See  for example  \cite{MohanKumar} for surface  rational double points   as well as  \cite{BrevikNolletCl}.
There are examples of isolated analytic singularities such that the local (analytic) class group cannot be reconstructed from the local (algebraic) class group of any algebraic approximation \cite{KollarLocPic} (in case of isolated singularities the class group and the local Picard group are identified (Definition \ref{localPic}).

\subsection{$\Q$-factoriality and the ``Calabi-Yau condition"}\label{CYCond}

  \begin{thm}[Koll\'ar, see also Proposition 3.5,  \cite{FujinoLogTer}]
 Let $X$ be a normal $\Q$-factorial scheme, $\tilde X$ a normal scheme and  $\phi: \tilde X \to X$ a birational morphism. Then any irreducible component of the exceptional locus has codimension one.   \end{thm}
 Note that the proof also works in the analytic setting.
 Combining Koll\'ar's Theorem and the previous $\Q$-factorialization results we have: 
 \begin{corollary}\label{QFal} Let $X$ be an algebraic  threefold  with klt singularities. There exists  a small  birational bimeromorphic morphism $\phi: \tilde X \to X$  
  if and only if $X$ is not 
 $\Q$-factorial. \end{corollary}
 
  \begin{corollary}\label{QFan} Let   $(\mathcal U^h, p)$ be a threefold analytic germ with at most an isolated klt singularity at $p$. There exists  a small  bimeromorphic morphism $\phi: \tilde U^h \to U^h$
   if and only if $(\mathcal U^h, p)$ is not analytically
 $\Q$-factorial. \end{corollary}
 If $X$ is a threefold with canonical singularities and $K_X \simeq \mathcal O_X$, then  a smooth resolution $\tilde X \to X$ with $K_{\tilde X} \simeq \mathcal O_{\tilde X}$ can exist only if either $X$ is canonical but not terminal or $X$ is not $\Q$-factorial.  It is often of interest in the physics literature to determine the existence of such resolutions, which are said to preserve 
 the ``Calabi-Yau condition", see for example \cite{ArrasGrassiWeigand}.

\section{Cohomologies, local and global (analytic) $\Q$-factoriality, Rational Poincar\'e}\label{ClGQFRP}


\begin{thm} [\cite{CGMP}, \cite{GM88}, \cite{MacPhersonICM}] Let $X$ be a  complex compact analytic threefold. The intersection cohomology  of $X$ has Poincar\'{e}  duality over $\Q$: 
$ {IH^i (X, \Q)}^{\vee} \simeq IH^{6-i} (X, \Q)$. 
\end{thm}

\begin{thm} [\cite{CGMP}, \cite{deCataldoMigliorini05},  \cite{PetersSteenbrink}; conjectured in \cite{MacPhersonICM}] Let $X$ be a  complex projective algebraic variety. The intersection coholomogy $ {IH^{k}(X, \Q)}$  has a pure Hodge structure of weight $k$,  in particular there is a  Hodge decomposition with  the property $ {IH^{i ,j}(X, \C)}=\overline{{IH^{j,i}(X, \C)}}$, $k=i+j$.  
\end{thm}
 If $X$ is smooth then the intersection cohomology equals the regular cohomology. This is also true for other types of  singular varieties, in particular for rational homology manifolds:

 \subsection{Rational homology manifolds}
 Recall that all varieties are assumed to be normal and connected.
 \begin{definition} A  complex threefold $X$  is  a rational homology manifold if and only if  for every point $p \in X$,  $H_6( X,  X \setminus p ;  \ \Q) = \Q$ and $ H _i( X,  X \setminus p ;  \ \Q) =0, \ \   i  \leq  5 $.\end{definition}

Orbifolds are examples of rational homology manifolds. We are interested in rational homology manifolds which are not orbifolds. In fact, an application and motivation of this paper is the study of  singular Calabi Yau threefolds, while three-dimensional Gorenstein orbifold singularities are smooth.

\begin{thm}[\cite{GoreskyMacPherson2, GM88}]\label{ratHomP} Let $X$ be a complex analytic variety   which is a rational homology manifold.  The intersection cohomology and  the ordinary  (simplicial) cohomology 
  coincide. In particular if $X$ is compact, Poincar\'e duality holds and  $\chi_{top}(X)$ can be computed with any of these theories.  \end{thm}

\begin{definition}  $L$ is  a rational homology $5$-sphere  if and only if   $H_0( L, \Q)=H_5( L, \Q)= \Q$ and  $ H _i( L , \Q) =0, \ \ 0 < i < 5 $. \end{definition}

Let $ p \in \mathcal U \subset \C^n$, $\dim_{\C} \mathcal U =3$, $D$ a suitable small ball around $p$, $\mathcal V_p = \mathcal U \cap D$, $S=\partial D$  and $L_p= \mathcal U \cap S$. $L_p$ is the link of $ p \in \mathcal U$. (If $p$ is an isolated singular point,  then $\mathcal V_p \cap D$ is a cone over $L_p$.)

Then, see for example \cite{LMaximNotes}:
\begin{proposition}\label{RHSrhm} $X$ is a rational homology manifold if and only if, $\forall  \ p \in X$, the link $L_p$ is a rational homology sphere.
\end{proposition}

\subsection{ Rational homology manifolds and  $\Q$-factoriality}

\begin{thm}[with J. Shaneson]\label{kltRHMT} Let $\mathcal U\subset \C^n$, $\dim_{\C} \mathcal U =3$ an algebraic (analytic) variety with rational singularities. Assume that $\mathcal U \setminus \Gamma$ is smooth and that for every $ p \in \mathcal U$,  $\pi_1(L_p)$ is finite. Then $\mathcal U$ is a rational homology manifold if and only if $\mathcal U$ is locally analytically $\Q$-factorial.
\end{thm}   
\begin{proof}  We will prove  that $L_p$ is a rational homology sphere, for all $p \in \mathcal U$  if and only if $(\mathcal U,p)$ is locally analytically $\Q$-factorial.  The statement then follows from Proposition \ref{RHSrhm}. By assumption, 
 ${H_1(L_p,\Q)=0}$. 
 
 $\mathcal U$ is complex and normal, thus $\codim \Gamma \geq 2$. Let $Z$ be the support of $\Gamma$ as a topological variety, $Z^{(2)}$ the singular locus of $Z$ and $Z^{(1)} \stackrel{def}= Z \ \setminus Z^{(2)}$.

\noindent \underline{Case 1} :  $p \in Z$   an isolated singularity. By restricting $\mathcal U$  let us  assume that $(\mathcal U, p)$ is a small neighborhood with an isolated singularity at $p$.
We can then follow the argument  of  \cite[Lemma 4.2]{KollarFlops} for terminal singularities. Since ${H_1(L_p,\Q)=0}$  by Poincar\'e duality we have ${H_4(L_p,\Q)=0}$. Because the singularities are rational, Flenner's result \cite[Satz 6.1]{Flenner}  implies  that $H^2(L_p, \Z) \simeq \Pic ^h(\mathcal U\setminus p)=\cl (\cO^h_{\mathcal U,p})$. Then $H_3(L_p, \Q)= H_2(L_p, \Q)=0$ if and only if the local class group $\cl (\cO^h_{\mathcal U,p})$ is torsion, that is, $(\mathcal U,p)$ is analytically $\Q$-factorial (Proposition \ref{localanclassgroup}).

\smallskip

\noindent \underline{Case 2} \footnote{We are grateful to M. Bies for providing us with  the TeX code for the figure.}:  $p \in Z^{(1)} $.  By suitably restricting $\mathcal U$ we can assume that $\mathcal U$ is a small open homotopically equivalent neighborhood of $\mathcal W \stackrel{def}=D^2 \times c \mathcal K$, where $c\mathcal K$ is the cone over a real 3-manifold $\mathcal K$, with cone point $C$ and $D^2 \subset \C^3$ a ball \cite{GM88}. Let $L_p$ be the link of $p$.

\vskip 0.17in

\begin{center}
\begin{tikzpicture}[scale=0.8]

    \def\l{4} 
    \def\h{1} 
    \def\s{-0.4} 

    \draw[dashed,blue] (0,0,0)--(0,\h,0);
    \draw[blue] (0,\h,0)--(\h,\h,0);
    \draw[blue] (\h,\h,0)--(\h,0,0);
    \draw[dashed,blue] (\h,0,0)--(0,0,0);
    \draw[dashed,blue] (0,0,0)--(\s,0,\l);
    \draw[blue] (\h,0,0)--(\h+\s,0,\l);
    \draw[blue] (\h,\h,0)--(\h+\s,\h,\l);
    \draw[blue] (0,\h,0)--(\s,\h,\l);
        
    \draw[orange] (\s+0.5*\h,0.5*\h,\l) -- (0.5*\h,0.5*\h,0);
    \draw[red,fill] (\s+0.5*\h,0.5*\h,\l) circle (0.04);
    \draw[black,fill] (0.5*\s+0.5*\h,0.5*\h,0.5*\l) circle (0.04) node[below right] {$P$};
    
    \draw[red] (\s,0,\l)--(\s,\h,\l)--(\h+\s,\h,\l)--(\h+\s,0,\l)--(\s,0,\l);
    
    \draw[gray] (\s+0.5*\h,0,\l)--(\s,0.5*\h,\l)--(\s+0.5*\h,\h,\l)--(\s+\h,0.5*\h,\l)--(\s+0.5*\h,0,\l);

    \node[red] at (1.5*\h,0,\l) [right] {\small{$K^3 \colon$ 3 manifold}};
    \node[gray] at (1.5*\h,0,0.666*\l) [right] {\small{cone over $k \colon \mathcal{A}$}};
    \node[blue] at (1.5*\h,0,0.333*\l) [right] {\small{$L_p \colon$ a link of $P$}};
    \node[orange] at (1.5*\h,0,0*\l) [right] {\small{$Z^{(1)} \cong D^2$}};

    \draw[red] (1.5*\h,0,\l) edge[out=150,in=-60,->] (0.8*\h+\s,-0.05*\h,\l);
    \draw[gray] (1.5*\h,0,0.666*\l) edge[out=160,in=-20,->] (\s+0.8*\h,0.25*\h,\l);
    \draw[blue] (1.5*\h,0,0.333*\l) edge[out=170,in=-20,->] (\h+0.3*\s,0,0.4*\l);
    \draw[orange] (1.5*\h,0,0*\l) edge[out=140,in=-30,->] (0.5*\h+0.1*\s,0.5*\h,0.2*\l);
    
\end{tikzpicture}
\end{center}

Then $\mathcal U \supset \mathcal W$ and
\bea\label{Lp}
L_p = \partial W =  ( \partial D^2 \times c\mathcal K) \cup (D^2 \times \mathcal K),
\eea
${ L_p    \setminus   Z \cap L_p} =  [\partial D^2 \times {(c\mathcal K \setminus C)}]\cup (D^2 \times \mathcal K)$.
The Mayer-Vietoris sequence for cohomology implies that $H^2(\mathcal W \setminus Z, \Z) \simeq H^2(L_p \setminus Z, \Z) = H^2(\mathcal K, \Z)$.
 Because the singularities are rational, Flenner's result \cite[Satz 6.1]{Flenner}  implies  that $H^2(L_p  \setminus Z, \Z) \simeq \cl (\cO^h_{{\mathcal U},p})$.
 Then $\cl (\cO^h_{{\mathcal U},p})$ is finite if and only $H^2(\mathcal K, \Q)=0$.  Since $\mathcal K$ is a manifold, Poincar\'e duality implies that ${H^2(\mathcal K, \Q)=0}$ if and only if $\mathcal K$ is a homology 3 sphere.  
We already remarked that $H_1(L_p, \Z)=0$; we now  use the  Mayer-Vietoris sequence for homology  and the decomposition in  (\ref{Lp})  to compute  $H_i(L_p,\Z),  \ i =2,3,4.$ 
 Then $H_i(L_p,\Z)=0,  \ i =2,3,4$ if and only if $H^2(\mathcal K, \Q)=0$, that is if and only if $\cl (\cO^h_{{\mathcal U},p})$ is finite. The statement  from $\mathcal W$ follows from 
 Proposition \ref{BerltramettiOdettiFossumStorch}.
 
\smallskip

\noindent \underline{Case 3} :  $p \in Z^{(2)} $.  Again, by suitably restricting $\mathcal U$, we can assume that $Z^{(2)} \cap \mathcal U =\{p\},$   and  \\${L_p= (L_p \setminus Z \cap L_p) \cup _\ell S^1_\ell}$, $ \ S^1_\ell \subset Z ^{(1)} \cap L_p$.
Write: 
\bea\label{Case3} L_p= (L_p \setminus Z \cap L_p) \cup _\ell N_\ell, 
\eea
where $N_\ell$ is open, it has the same homology of $S^1$ and $N_\ell \to S^1$, with fiber $c\mathcal K$, a cone on  3 manifold $\mathcal K$. Take  $\ q \in S^1_\ell $, $ \mathcal W_q \stackrel{def}=D^2 \times c \mathcal K $  to be a small neighborhood of $q$ in $\mathcal U$ \cite{GM88}. 

\noindent {\it Claim 1: $L_p$ is a rational homology manifold if an only  if $\mathcal W_q$ is locally analytically $\Q$-factorial, $ \forall q  \in Z \cap L_p$.}\\
{\it Proof of Claim 1.} Recall that  $q \in Z^{(1)}$, as in Case 2.  We will prove that $\mathcal W_q \cap L_p$ is a rational homology sphere if and only if   $\mathcal W_q$ is locally analytically $\Q$-factorial, for all $ q$. In fact  the previous description implies that  $ \mathcal W_q \cap L_p= I \times c \mathcal K $ is a neighborhood of $q$ in $L_p$, hence $ \mathcal W_q \cap L_p$ is a rational homology sphere if and only if  $\mathcal K$ is a rational homology sphere \cite{GM88}.  From Case 2 we know that $ \mathcal K$ is a rational homology sphere if and only if $\mathcal W_q$ is analytically $\Q$-factorial. 


\noindent {\it Claim 2: If $L_p$ is a rational homology manifold $H_4(L_p, \Q)=0$.}\\
{\it Proof of Claim 2.}  Poincar\'e duality.

\noindent {\it Claim 3: $H_2(L_p, \Q)=0$,  if an only if $\mathcal U$ is locally analytically $\Q$-factorial. }\\
{\it Proof of Claim 3.}  To compute $H_2(L_p,\Q)$ we  use again  the Mayer-Vietoris sequence for homology  and  the decomposition in (\ref{Case3}). 
Note that  for each $q$, $(L_p \setminus Z \cap L_p) \cap  N_q$ is homologically a bundle over $ S^1$, with fiber $\mathcal K$. Then $H_2(\cup N_q, \Q)=0$; recall that $H_2(L_p \setminus Z \cap L_p, \Q)=0$ if and only if $\cl (\cO^h _{\mathcal U, p})$ is finite \cite{Flenner}.  We conclude again by
 Proposition \ref{BerltramettiOdettiFossumStorch}.

\noindent {\it Claim 4:  If $L_p$ is a rational homology manifold and $H_2(L_p, \Q)=0$ then $H_3(L_p, \Q)=0$. }\\
{\it Proof of Claim 4.} By Poincar\'e duality.
\end{proof}

 \begin{thm}\label{klt} Let $\mathcal U\subset \C^n$, $\dim_{\C} \mathcal U =3$ an algebraic  threefold with klt singularities. Then $\mathcal U$ is a rational homology manifold if and only if $\mathcal U$ is locally analytically $\Q$-factorial.
\end{thm}   

\begin{proof}   In fact $\pi_1(L_p)$ is finite \cite[Cor~1.5]{TianXuFund}. The statement then follows from Definition \ref{laqf}, Proposition \ref{RHSrhm} and Theorem \ref{kltRHMT}.
\end{proof}

Banagl shows that a threefold $X$  with canonical singularities, trivial canonical divisor and $h^1(X,\cO_X)> 0$ is a rational homology manifold \cite[Remark 6.4, page 29]{Banagl2017}. The threefolds  in  the examples \ref{general} and \ref{terminalconifolds}, a rational homology manifold and one which is not,  below can  be taken to be Calabi-Yau (trivial canonical divisor and $h^1(X,\cO_X)= 0$).

  \begin{ex} \label{A_a-1sing}(Rational homology manifolds, non rational homology manifolds, local) \phantom{XXX} 
  
    Let  $(\mathcal U, p)$ be an
 $A_{a-1}$  Kleinian threefold singularity, that is $(\mathcal U, p)$ is  the zero-locus of\\
$ \quad
f(z,x_1,x_2,x_3) = z^a + x_1^2 + x_2^2 + x_3^2 \quad {\rm with} \quad a \geq 2 \in \mathbb N \,.$ 
These are terminal (and non-canonical) singularities \cite[Th.~1.1]{Reid}. A local, small resolution is possible if and only if $a$ is even, \cite[Cor.~1.6]{Reid}, \cite{Atiyah, Brieskorn}. Then
Corollaries \ref{QFal} and \ref{QFan}   imply that $(\mathcal U, p)$ is $\Q$-factorial (locally analytic $\Q$-factorial) if and only if $a$ is odd.
On the other hand one can also find directly that $(\mathcal U,p)$ is a rational homology manifold if and only if $a$ is odd     \cite{BrieskornBeiS}, \cite[Theorem 4.10]{DimcaBook}.
Flenner  \cite{Flenner}  proves the statement directly using the local divisor class group  
(see Section \ref{FQFAF}).

  \end{ex}
   Both these types of examples occur  in examples of  (elliptically fibered)   threefolds with $\Q$-factorial terminal singularities:
   
\begin{ex}\label{general} (A threefold with $\Q$-factorial terminal singularities and a rational homology manifold.) Let $ \pi: X \to B$ be an elliptic fibration in Weierstrass form with  general singular Kodaira fibers  of type $II$ (cusps) and $I_1$  (nodes) over the smooth points of the discriminant of the fibration. Then the singularities of $X$ are $\Q$-factorial, terminal, but not smooth, with local equation $z_0^a + \sum_{i=1,3}z_i^2=0$, with a odd. These  singularities are analytically $\Q$-factorial  and $X$ is a rational homology manifold. This is  Case 2 in the proof of Theorem \ref{Anomalycondition1}, see also \cite{ArrasGrassiWeigand}.
\end{ex}

\begin{ex}\label{terminalconifolds} (A threefold with $\Q$-factorial terminal singularities but not a rational homology manifold.) Let $ \pi: X \to B$ be a general elliptic fibration in Weierstrass form with singular Kodaira fibers  of type $I_1$ over the smooth points of the discriminant of the fibration such that the singular points of the discriminant are cusps and simple normal crossing divisors.
Then the singularities $ p \in X$ occur over the simple normal crossing points of the discriminant; these singularities are ordinary double points (``conifold") terminal $\Q$-factorial, but not smooth, with local equation $z_0^2+ \sum_{i=1,3}z_i^2=0$. These are not analytically $\Q$-factorial, and $X$ is not a rational homology manifold.  The Jacobian elliptic fibration of a general  genus one fibration has exactly this type of singularities  see for example \cite{BraunMorrison, Morrison:2016lix}.
This is  in fact Case 1 in the proof of Theorem \ref{Anomalycondition1}, see also \cite{ArrasGrassiWeigand}.

\end{ex}

In the next Section we show that nevertheless threefolds with isolated klt  singularities satisfy some Poincar\'e duality over the rationals.

\subsection{Rational Poincar\'e Duality}

 \begin{remark}\label{R3}  $\bar R_3= E^3 / <-\omega I_3> \ $  is an example of a  ``Calabi-Yau" threefold with isolated klt but not canonical singularities  \cite{OguisoCYkltRationality}. Here $E  = \C /  ({\Z + \omega \Z})$ and  $\omega= e^{\frac{2}{3}\pi i }$. The canonical divisor is numerically trivial, $h^1(\bar R_3,\cO_X)=h^2(\bar R_3,\cO_X)=0$,
it is elliptically fibered  and it is a rational homology manifold.  $\bar R_3$ is  also rational \cite{OguisoTruong}. Klt varieties with numerically trivial canonical divisors have Bochner's type properties \cite{GrebGuenanciaKebekus}.
\end{remark}

More generally, we can prove

 \begin{thm}\label{rPD1}Let $X$ be a projective algebraic threefold with isolated klt singularities. 
 \begin{enumerate}
\item The cup product  with the fundamental class gives an isomorphism \\
${H^5(X, \Q ) \xrightarrow{\sim} H_1(X, \Q)}$, that is $H_5(X, \Q)$ and $H_1(X, \Q)$ are Poincar\'e duals over the rationals.
\item If in addition  $b_2(X)=b_4(X)$, $X$ satisfies rational Poincar\'{e} duality, i.e.  the cup product  with the fundamental class gives Poincar\'{e} duality  with rational coefficients. \end{enumerate}
 \end{thm}
\begin{proof} If  $X$ is a rational homology manifold, both  statements follow from Theorem \ref{ratHomP} (no assumption of $b_2(X)=b_4(X)$ is needed).  
If $X$ is not a rational homology manifold, then it is not locally analytically  $\Q$-factorial  by Theorem \ref{klt}. Let $\{({\mathcal U},P)\}$  be  a collection of  contractible neighborhoods of the isolated singular points $\{P\}$ which are  not  analyticaly $\Q$-factorial; without loss of generality we can assume that any two neighborhoods $(\mathcal U,P)$  do not intersect. 
Let $ \phi_P:  ({\mathcal U'}_P, \Gamma_P) \to  ({\mathcal U_P},P)$ be the $\Q$-factorialization, that is a small  bimeromorphic  morphism, an isomorphism in codimension 1, with  $({\mathcal U'}_P, \Gamma_P)$ analytically $\Q$-factorial   and klt (Theorem \ref{Klqf}). ${\mathcal U'}_P$ is simply connected \cite{TakayamaSimpleConn}\footnote{Shepperd-Barron proved in an unpublished note in 1989 that   ${\mathcal U'}_P$ is simply connected when $P$ is a canonical singularity.} and  ${\mathcal U'}_P$  is contratctible to  the exceptional locus $\Gamma_P$. 
  Hence $H_1 ({\mathcal U'}_P)=H_3 ({\mathcal U'}_P)=H_4 ({\mathcal U'}_P)=H_5 ({\mathcal U'}_P)=0$.  
    Let $X'$ be the complex analytic threefold obtained by patching in $X \setminus \cup \{ {\mathcal U}_P\}$ the collection of the singular neighbhoroods   $\{ {\mathcal U'}_P\}$.   Let
  $f: X' \to X$ be the induced morphism. 
  The  commutative diagrams obtained by combining the Mayer-Vietoris sequences in homology and cohomology for  $X= (X \setminus \cup _P  {P} ) \coprod _P {\mathcal U}_P$ and  $X'= (X \setminus \cup _P{\Gamma_P }) \coprod _P {\mathcal U'}_P$ 
 imply that:\\
  $H_1(X) \simeq H_1(X') $, $\ H_5(X) \simeq H_5(X')$, $ \ H^4(X) \simeq H^4(X') \ $, and $H^5(X) \simeq H^5(X')$.   
  
Theorems \ref{kltRHMT} and  \ref{klt} show that $X'$ is a rational homology manifold, then  Poincar\'{e} duality on $X'$  holds and the top arrows in the diagrams below are isomorphisms:

\begin{equation*}
 \xymatrix{
   H^5(X') \ar[r]^{\cap [X']} &   H_1(X')  \ar@<-2pt>[d]_{\simeq} \\
  H^5(X) \ar@<-2pt>[u]_{\simeq} \ar[r]^{\cap {[X]}} &   H_1(X) 
}
 \ \quad \   \quad \ \   \quad   \quad 
 \xymatrix{
   H^4(X') \ar[r]^{\cap [X']} &   H_2(X')  \ar@<-2pt>[d]_{f_*} \\
  H^4(X )\ar@<-2pt>[u]_{\simeq} \ar[r]^{\cap {[X]}} &   H_2(X)  
}
\end{equation*}
The first statement follows immediately. 
The  Mayer-Vietoris sequence implies also that ${f_*:H_2(X') \to H_2(X)}$ is surjective, and thus, if $b_2(X)=b_4(X)$,   ${H^4(X) \stackrel{\cap {[X]} }\rightarrow H_2(X)} $ is an isomorphism, as a surjective morphism between spaces of the same dimension.
\end{proof}

 \begin{proposition}\label{NS} Let $X$ be a projective
  $\Q$-factorial threefold with  isolated rational hypersurface singularities and $h^2(X,\cO_X)=0$, then $b_2(X)=b_4(X)$.
 \end{proposition}
 
\begin{proof}  It follows from \cite[Theorem 3.2]{NamikawaSteenbrink}.
\end{proof}

\begin{corollary}\label{kltI}  Let $X$ be a   projective $\Q$-factorial threefold with klt isolated hypersurface singularities  and $h^2(X,\cO_X)=0$. Then $X$ satisfies rational Poincar\'{e} duality.
\end{corollary}
\begin{proof} It follows from Theorem \ref{rPD1} and Proposition \ref{NS}.
\end{proof}

\begin{thm}\label{rPD2} Let $X$ be a projective   Gorenstein $\Q$-factorial threefold with terminal singularities and $h^2(X,\cO_X)=0$. Then $X$ satisfies rational Poincar\'{e} duality.\end{thm}
\begin{proof} In fact, threefold Gorenstein  terminal singularities are isolated rational hypersurface singularities (see Theorem \ref{kltRat}   and Remark \ref{cdvHyp}). The statement follows from the previous Corollary \ref{kltI}. \end{proof}


\begin{corollary} Let $X$ be a projective  minimal 
 threefold
  of Kodaira dimension 0.  Then $X$ satisfies rational Poincar\'{e} duality. 
\end{corollary}
\begin{proof} If  $h^1(X,\cO_X) \neq 0$  then $X$ is a rational homology manifold \cite[Remark 6.4, page 29]{Banagl2017}.
If $h^1(X,\cO_X) = 0$  then $X$ is either a Calabi-Yau or a quotient of a Calabi-Yau by a finite group and the statement follows from Theorem  \ref{rPD2}.
 \end{proof}
In particular:
\begin{corollary}\label{chiklt}   Let $X$ be  a projective  $\Q$-factorial  Calabi-Yau  or  Gorenstein  Fano threefold with terminal singularities. Then $$\chi_{top}(X)=2 \{ 1
+ b_2(X)\} -b_3\,.$$
\end{corollary}
\begin{proof} In fact $X$ satisfies rational Poincar\'{e} duality. 

\end{proof}

\section{The third Betti number  (and complex deformations of Calabi-Yau threefolds)}\label{cxDefMT}
 
  \subsection{Milnor and Tyurina numbers}

 Let
    $(\mathcal U, 0)\subset \C^{n+1}$ be a neighborhood of an isolated hypersurface singularity $P=0$,  defined by ${f=0}$.
\begin{definition}The Milnor number of $P$  can be defined as
\begin{align*}
m(P) =\dim_{\mathbb C}( \mathbb C\{ x_1, \ldots, x_{n+1} \}/{ \small{<  \frac{\partial f}{\partial x_1}, \ldots,   \frac{\partial f}{\partial x_{n+1}}  > }}) \,.
\end{align*}
\end{definition}
\begin{definition}
The {Tyurina number}  $\tau(P)$  is the dimension of the space of versal deformations of the hypersurface singularity at $P$ in $\mathcal U$ and it is computed algebraically as
\begin{align*}
\tau(P) = \dim_{\mathbb C} (\mathbb C\{ x_1, \ldots, x_{n+1} \}/< f, \frac{\partial f}{\partial x_1}, \ldots,   \frac{\partial f}{\partial x_{n+1}}  >).
\end{align*}
\end{definition}
\begin{remark}\label{Saito}$m(P) \geq \tau(P)$ and 
Saito  proved that $\tau(P) = m(P)$ if and only if  $P$ is a weighted hypersurface singularity  \cite{SaitoK1971, Looijenga2013}.
 Saito's  Theorem has been generalized to complete intersections by Greuel \cite{Greuel1980}.
\end{remark}
 $m(P)$ and $\tau(P)$ are also  computable by SINGULAR \cite{GreuelLossenShustin2007}  and Maple \cite{RossiTerracini}.

\subsection{Complex deformations and $b_3(X)$} 

  \begin{proposition}\label{b3}   Let $X$ be a $\Q$-factorial Calabi-Yau threefold with  terminal singularities  and  let $\cD(X)$ denote the dimension of the Kuranishi space  of $X$, the space of complex deformations of $X$. Then
\begin{equation*}
\cD(X)=\cD(X_t) =   \frac{1}{2} b_3(X_t) -1 =    \frac{1}{2} \{b_3(X)+ \sum_{ P} m(P) \}-1,
\end{equation*}
  where $X_t$ is the smoothing of $X$ and $m(P)$  the Milnor number of  the singular point $P$.
   \end{proposition}
    \begin{proposition}\label{b3T}   Let $X$ be a $\Q$-factorial Calabi-Yau threefold with  terminal singularities which are weighted hypersurface singularities and  let $\cD(X)$ denote the dimension of the Kuranishi space  of $X$, the space of complex deformations of $X$. Then
\begin{equation*}
\cD(X)= \frac{1}{2}\{b_3(X)+ \sum_{P} \tau(P) \}-1,
\end{equation*}
  where  $\tau(P)$  is the  Tyurina number of  the singular points $P$.
  
  \end{proposition}

  \begin{proof}[Proof of Propositions \ref{b3} and \ref{b3T}] The singularities are hypersurface singularities since they are terminal of index 1.
Proposition \ref{b3}  and Proposition \ref{b3T} follow from Theorems 1.3 and  3.2 in \cite{NamikawaSteenbrink} and from Remark \ref{Saito} above \cite{SaitoK1971}. Theorem 1.3 proves that a $\Q$-factorial Calabi-Yau threefold with  terminal singularities admits a smoothing to a Calabi-Yau $X_t$. Theorem 3.2  also proves that if  a  threefold $X$ with  isolated rational hypersurface singularities has a smoothing $X_t$ and ${h^2(X,\cO_X)=0}$, then
$b_3(X)= b_3(X_t) -  \sum_{\text{sing } P} m(P) $, where $m(P)$ are the Milnor numbers of the singularities.
\end{proof}

\begin{remark}Proposition \ref{b3}  implies that $\frac{1}{2} b_3(X) + \frac{1}{2}\sum_P m(P) \in \mathbb Z$, but in general $\frac{1}{2} b_3(X) \notin \mathbb Z$ because the Hodge decomposition and Hodge duality may not hold.
The Kleinian hypersurface singularities  of type $A_{a-1}$ of  Example \ref{A_a-1sing} are weighted hypersurface singularities,  and thus ${m(P) = \tau(P) = a-1} $.
They are rational homology manifolds  if and only if $a$ is odd, in which case both $b_3=2 h^{1,2}$ and $a-1$ are even.
\end{remark}

\begin{remark}  \label{remarklocalized}
If  $X$ is a Calabi-Yau variety with $\Q$-factorial terminal singularities, the dimension of the space of complex deformations splits into a "localized" and  ``non-localized" contribution, given by the dimension of the versal deformation space of the singularities and, respectively,  the remaining deformations:
\begin{align}\label{CdefY3a}
\cd(X) = \underbrace{-1+ \frac{1}{2}\{ b_3(X) + \sum_P ( {m(P)-2\tau(P)}) \}}_{non-localized}+  \underbrace{\sum_P {\tau(P)}}_{localized} \,.
\end{align}
 In Section  \ref{FtheoryInterpretation} we present a natural interpretation  of this decomposition for physics.
\end{remark}
 The  decomposition  (\ref{CdefY3a}) suggests the existence of a ``local-to-global principle" for deformations of Calabi-Yau threefolds with $\Q$-factorial terminal singularities:
 \begin{conjecture}\label{localtoglobaldef}
There exists a natural decomposition of the Kuranishi space of $X$ into the space of complex structure deformations of $X$ which deform the isolated singularities, whose dimension is the sum of the dimensions of the versal deformations (Tyurina numbers), and  the remaining space of deformations of $X$ which do not change the location or form of the isolated singularities. 
\end{conjecture}
Note also that in the general hypothesis considered in Section \ref{chitop} and \cite{ArrasGrassiWeigand} $m_P= \tau_P$.

\section{The second Betti number  (and K\"ahler deformations);  topological Euler characteristic}\label{b2}

 Let   $X$  be a (normal) complex threefold with ${h^2(X,\cO_X)=0}$. The exponential sequence, see for example \cite[pg. 142]{GrauertRemmert}, implies that $b_2(X)$ is the rank of the N\'eron-Severi group, namely $b_2(X)=\kD(X)$ (Section \ref{motivation}).

More generally  the following holds:

\begin{thm}[Srinivas, see Appendix]\label{srinivasth} Let X be a normal projective variety over the field $\C $ of complex numbers. Let $\pi : Y  \to X $ be a resolution of singularities. Assume $R^1 \pi_*(\cO _Y)= 0$ (this condition is independent of the choice of resolution). Then
\begin{enumerate}
\item  the singular cohomology $H^2(X,\Z) $ supports a pure Hodge structure;
\item  the  N\'eron-Severi group, 
 $\ns(X)\stackrel{def}=c_1 (Pic(X)) \subset  H^2(X,\Z)$ coincides with the subgroup of (1, 1) classes, i.e. with the subgroup\\
$\{\alpha \in H^2(X,\Z) | \alpha_\C  \in H^2(X,\C)\text{  is of type } (1,1)\}.$
\end{enumerate}
\end{thm}
The above Theorem is used in the Section \ref{BG}.

\begin{corollary}\label{chiGor} Let $X$ be a projective  
$\Q$-factorial
 threefold with isolated klt  hypersurface singularities and $h^2(X,\cO_X)=0$. Then 
$$\chi_{top}(X)=2 \{ 1- b_1(X)+ \kD(X) \} -b_3(X) \,.$$
\end{corollary}

\begin{proof} In fact  $b_2(X)=\kD(X)$,  as we observed at the beginning of  this Section. The statement then follows from Corollary \ref{chiklt}.
\end{proof}


\begin{corollary}\label{chirat} Let $X$ be a projective threefold with $h^2 (X, \cO_X)=0$ and $\Q$-factorial rational singularities which are analytically $\Q$-factorial, then 
\bea
\chi_{top}(X)&=& 2 - 4  h^{1,0} (X) +2 \kD (X)- b_3(X) \\
&=& 2 - 2 h^{0,3}(X) - 4  h^{1,0} (X)+ 2 \{\kD (X)-h^{1,2}(X)\}. 
\eea
\end{corollary}

 \begin{thm} \label{main} Let $X$ be a Calabi-Yau threefold with $\Q$-factorial terminal singularities. Then
$$\chi_{top}(X)=2 \{\kD(X)-\cD(X)\}+ \sum_P m(P),$$ where $m(P)$ is the Milnor number of the singular point $P$.
\end{thm}

\begin{proof} The statement follows from Proposition \ref{b3} and Corollary \ref{chiGor}.
\end{proof}

\begin{corollary}\label{mainbir}  All the terms in the equation of Theorem \ref{main} do not depend on the choice of  minimal model $X$.
\end{corollary}
\begin{proof} In fact, let $X$  and $X'$  be birationally equivalent minimal  threefolds with $\Q$-factorial terminal singularities. Then $b_j(X)=b_j(X')$ $\forall j \ $ \cite[Theorem 3.2.2]{KollarFlipsFlopsMMetc}.   Since birationally equivalent minimal models are related by flop transitions, 
$ \kD(X)=\kD(X')$. Also, $X$ and $X'$ have the same analytic type of singularities  \cite[Theorem 2.4]{KollarFlops}; hence they also have the same Milnor numbers. Furthermore the dimensions of the miniversal deformation spaces are the same \cite[Theorem 12.6.2]{KollarMoriFlips}. 
\end{proof}

\section{Gauge algebras and representations (a Brieskorn-Grothendieck Program)}\label{BG}
 It was noted by Du Val and Coxeter \cite{DV, coxeter:annals}
that rational double points
are classified by the Dynkin diagrams of the
simply laced Lie algebras of type $\mathfrak a_n, \mathfrak d_n, \mathfrak e_6, \mathfrak e_7, \mathfrak e_8$.
 In fact, if we resolve the singularity by
blowing up,
the dual diagram of the exceptional divisors
is one of the above Dynkin diagrams.  Further connections between the Lie algebras and surface singularities were discovered in  works by Brieskorn, Grothendieck,  Tyurina,  and Slodowy. A mathematical explanation  of the parallelisms of the classification remains elusive.
String theory provides a framework in which a Lie algebra  $\mathfrak g$, the \emph{``gauge algebra"},  is naturally associated to an elliptic fibration between Calabi-Yau manifolds.
All the Dynkin diagrams, including the non-simply laced ones, occur. 
  Deep relations between $\mathfrak g$, its representations and the geometry of the fibration have emerged, however  the correspondence is often case by case and  the assumption of smoothness imposes restrictions.
 In this section we review and state the expected correspondence in mathematical terms,  and extend it to  singular varieties, in particular to $\Q$-factorial terminal singularities.

In Sections \ref{nagasT}  and \ref{nagasM} we construct the algebras, in Sections \ref{unloc}   and  \ref{sec_localized}  we define a map between the codimension one and two strata of the discriminant locus of  the fibration and the representations of the algebra.
Theorem \ref{Anomalycondition1} in the following Section  \ref{chitop}  provides evidence for a ``Brieskorn-Grothendieck Program", associating Lie algebras and their representations to singularities of varieties.

Although Calabi-Yau threefolds with terminal singularities  and elliptic fibrations are the focus of the applications in Section \ref{chitop}, in this Section we state definitions and results in more generality.

\begin{definition} A genus one  fibration is a morphism $\pi : X\rightarrow B$ whose fibers over a dense set are genus one
curves. The complement of this dense locus is the discriminant of the fibration and it is denoted by $\Sigma $. 
$X$ is relatively minimal if $K_X \cdot \Gamma \geq 0$, for all the curves $\Gamma$ contracted by $\pi$ (or equivalently $K_X$ is $\pi$-nef).

 If  a genus one fibration  $\pi$ has a section $\sigma: B \to X$, it is called an elliptic fibration. 
\end{definition}

The support of the discriminant locus of a $X \to B$ a genus one fibration has  a stratified structure given by its singularities; we analyze this in Section \ref{subsection_reps}.
If  $X$  and $B$ are smooth and $\pi$ is an elliptic fibration, $X$ is the resolution of  the Weierstrass model $W$ of the
fibration,
 which has Gorenstein singularities \cite{Na88}:
\begin{definition} A Weierstrass model $W$ is defined by
$ y^2z-(x^3+\alpha xz^2+\beta z^3) = 0$  
where
$\alpha,\beta$ are  sections of ${\mathcal L}^{\otimes 4}$ and ${\mathcal L}^{\otimes 6}$ with ${\mathcal L}$ a line bundle on $B$.
\end{definition}
If $\dim W=2$,  the singularities of $W$ are the rational double points.  
If ${\mathcal L} = {\mathcal O}(-K_B)$, $K_W \simeq \mathcal O_W$. We are mostly interested in Weierstrass models $W$ which are birationally Calabi-Yau varieties.
By a rescaling of the Weierstrass equation, possibly together with a suitable blowup of the base,  we assume that  $\alpha$ and $\beta$  nowhere vanish simultaneously  with multiplicity equal to or higher than $4$ and $6$,
 that is there there are no ``non-minimal" points.  
The assumption is  necessary for the existence of an  equidimensional birationally equivalent elliptic fibration  $X \to B$; $X$ is a relative  minimal model   of $W \to B$. The condition is  also sufficient when $\dim W=3$ \cite{GrassiEqui}.

 Assuming $X$ to be smooth imposes restrictions.  However, without loss of generality we can still assume that $B$ is smooth if $\dim (X)=3$ \cite{Na88, Grassi1991}. 

\subsection{ (Gauge) algebras 
 and  the codimension one strata of the discriminant}\label{subsection_algebras} 
 
 \subsubsection{The abelian components of the gauge algebra}\label{agas}
If the Mordell-Weil group of the elliptic fibration has rank $r>0$, the gauge algebra includes an abelian part $\frak{u}(1)^{\oplus r}$. We briefly discuss the Mordell-Weil group and the abelian part of the gauge algebra after Theorem \ref{Anomalycondition1}; in the present work we focus on the  non-abelian part of the gauge algebra.

Next we present two methods to describe the non-abelian gauge algebras associated to the fibration, the first one uses the existence of a section. One novelty in our analysis is also the presence of singularities on $X$.

\subsubsection{The non-abelian components of the gauge algebra,   through the ``Tate algorithm"}\label{nagasT}~\vskip 0.1in

 The proofs of the following Lemmas are along the general arguments of \cite{GrassiMorrison11}, but  there they are not always stated  explicitly.  

\begin{proposition}\label{TateAlg}Let  $B$ be smooth. To  a Weierstrass model $W \to B$ 
there is a naturally associated 
Lie algebra $\mathfrak g  = \bigoplus_{\Sigma_j} \mathfrak{g}(\Sigma_j)$, 
where the sum is taken over the irreducible components of the discriminant locus. For any irreducible  component $\Sigma_j$,  the Kodaira fiber over the general point of $\Sigma_j$ and the possible
   $\mathfrak{g}(\Sigma_j)$ 
are listed in  the second and third column of Table \ref{tab:A}.
\end{proposition}

 \begin{proof}  The singular fibers for smooth (relatively) minimal elliptic surfaces were classified by Kodaira \cite{MR0184257}, and the associated algebra is the one associated to the rational double point of the singular Weierstrass model. For Weierstrass threefolds which are equisingular along all components, this association gives   the simply-laced algebras  of Table \ref{tab:A}. By a careful analysis of Tate's algorithm  \cite{MR0393039}  we extend the association  to  all dimensions.  The correspondence does not use the existence of a smooth (relatively) mimimal model of  the Weierstrass model $W \to B$.
The Kodaira classification and Tate's algorithm only depend on
the generic structure  of the elliptic fibration along  each irreducible component.
 The  Kodaira-Tate algorithm  as elaborated in   Appendix B of \cite{GrassiMorrison11} can still be applied, since $B$ is smooth and each irreducible component of the discriminant $\Sigma$ is a Cartier divisor, which is locally principal.
The modified algorithm  starts by constructing resolutions of the general singularities of the Weierstrass model $W$, which are the singularities over the general points of the discriminant locus.
 Then the analysis along the irreducible components  
$\Sigma_j$ of
the discriminant locus determines the algebra, together with a possible associated ``monodromy" which leads to the non-simply laced algebras.   \end{proof}

The modified Tate algorithm  also describes the structure of the partially resolved fibration
near each component:

\begin{lemma}\label{def:g'} Let  $B$ be smooth, $W \to B$ be a Weierstrass model and $W_{\Sigma_j} \to B$  the partial general resolution in the Proof of Proposition \ref{TateAlg}. Let $D^l_j$ be an irreducible Weil
divisor which maps surjectively onto $\Sigma_j$.   Then  the elliptic fibration induces on $D^l_j$  the structure of a  surface {generically} ruled either over $\Sigma_j$ or over $\Sigma_j'$,  a finite branched cover of $\Sigma_j$.
 Let $\ell_{j,l}$ be the general fiber of the ruling. In the non-simply laced case, the cover is of degree $2$ unless ${\mathfrak g}(\Sigma_j)={\mathfrak g}_2$;
in this case, the degree of the cover is $3$.
\end{lemma}

Note that $D^l_j$ is not always normal and also that $\ell_{j,l}$ is not always reduced or irreducible.
\subsubsection{The non-abelian components of the gauge algebra,   through the intersection matrix}\label{nagasM}~\vskip 0.1in

From now on we consider genus one fibrations of threefolds, although most of what we write can be generalized to higher dimensions and for more general fibrations:

\begin{definition}\label{Lambda0}  Let $Y$ be a threefold and assume  that $ H_2(Y, \Q) $ and  $H_4(Y, \Q) $ are Poincar\'e dual. Let
 $<,>: \ H_2(Y, \Q)  \times H_4(Y, \Q) \to \Q$ be the induced non-degenerate pairing.  
For $E \in H_2(Y, \Q)$, let ${[E]^\perp= \{ D \in H_4(Y, \Q) \text{ s. t. } <E,  D>=0\}}$.
For $B_1 \in H_4(Y, \Q)$, let ${[{B_1}]^\perp= \{ C  \in H_2(Y, \Q) \text{ s. t. } <C,  B_1>=0\}}$.
\end{definition}

We also denote by $[E]^ \perp$ its dual in $H^2(Y, \Q)$. 
When $Y$ is $\Q$-factorial the pairing is the intersection pairing between curves and $\Q$-Cartier divisors. 
\begin{definitionp}\label{Lambda}

Let $\pi: Y \to B$ be a  fibration with general  fiber $E$ and  assume that  $H_2(Y, \Q) $ and  $H_4(Y, \Q) $ are Poincar\'e dual. 
Assume also  that $B$ and $Y$  are $1$-rational.  \\
  Then $\pi^* ( {\ns (B)}) \subseteq H^{1,1}(Y, \Q) \cap H ^2(Y, \Z) \cap [E]^\perp $.

Set  $  \overline \Lambda  \stackrel{def}= H^{1,1}(Y, \Q) \cap H^2(Y, \Z) \cap [E]^\perp \ $ and   $\ \ \Lambda \stackrel{def}=    \overline \Lambda / \pi^*(\ns(B)) $.  
\end{definitionp}
\begin{proof} Note that  $\pi^* ( {\ns (B)}) \subseteq \ns(Y) \cap [E]^\perp$.  Since $B$ and $Y$ are $1$-rational (Section \ref{basics}) $\ns(Y) \subseteq  H^{1,1}(Y, \C) \cap H^2(Y, \Z) $, by Theorem \ref{srinivasth}. \end{proof}

 If  $\pi: Y \to B$ is a  
 factorial  relatively minimal threefold (with terminal singularities)  of a Weierstrass fibration $ W \to B$, with $B$ smooth as in Section \ref{nagasT},
  $\Lambda$ is generated by the exceptional divisors $D^l_j$ described in Lemma \ref{def:g'}, because the fibration is  equidimensional \cite{GrassiEqui}. The identification depends on the choice of the Weierstrass model or equivalently on the choice of the section of the fibration $\pi$. 
 
\begin{definition}\label{L1}
Let  $\pi: Y \to B$ be a fibration,  with general fiber   $E$ and $Y$ $\Q$-factorial.   Let $H_2^{(\pi)}(Y, \Q)$ be the span in  $H_2(Y, \Q)$ of $ \NE(Y/B)$, the classes of effective
 curves contracted by $\pi$,
 that is the set  of $\ell$ such that $\pi_*(\ell)=0$.   Let $L_1$  be $H_2^{(\pi)}(Y, \Q)$ modulo the  numerical equivalence class of $E$.   \end{definition}
 
\begin{proposition} Let  $Y$, $B$  and $\pi: Y \to B$ as in Definitions \ref{Lambda} and \ref{L1}. Then
the  Poincar\'e  pairing  induces an integral pairing when restricted to the classes of algebraic curves in $L_1$ and $\Lambda$.

\end{proposition}

$L_1$ and $\Lambda$ can be defined also for genus one  fibrations. When there is a section we have:

\begin{definition}\label{L2}  Let  $\pi: Y \to B$ be an elliptic fibration  with  section  $B_1 \simeq B$. Assume that $ H_2(Y, \Q) $ and  $H_4(Y, \Q) $ are Poincar\'e dual. Let $[B_1]^\perp$ be the orthogonal complement within $H_2(Y, \Q)$. Let $H_2^{(\pi)}(Y, \Q)$ be the span in  $H_2(Y, \Q)$ of  the effective curves contracted by $\pi$, $\bar L \subset H_2(Y,\Q)$ the subspace spanned by  the $\ell_{j,l}$ in $H_2^{(\pi)}(Y, \Q)$ and  ${L_2=H_2(Y,\Z) \cap \bar L \cap [B_1]^\perp}$.

\end{definition}

   If  $\pi: Y \to B$ is a  $\Q$-factorial  relatively minimal model   (with terminal singularities) of a Weierstrass fibration $ W \to B$,  as in Section \ref{nagasT},  $L_1 \simeq L_2$ and the isomorphism between the two definitions depends on the choice of a section; in this case we write $L \stackrel{def}= L_1 \simeq L_2$.
   
\begin{corollary} Let $X \to B$ be a genus one threefold with  Gorenstein $\Q$-factorial terminal singularities and $h^2(X,\cO_X)=0$. Then   $L_1$ and $\Lambda$ are well-defined and the  pairing  is   integral when restricted to $L_1$ and to $\Lambda$.
\end{corollary}

 $L_1$ and in particular $\Lambda$ in the statement are well defined if $X$ has isolated
 klt singularities and $b_2(X)=b_4(X)$, by  Theorem \ref{rPD1}.
 \begin{proof} Theorem \ref{rPD2}  implies that  $ H_2(X, \Q) $ and  $H_4(X, \Q) $ are Poincar\'e dual.  We noted in Section \ref{tckf}  that a Gorenstein threefold with $\Q$-factorial terminal singularities is actually  factorial.   \end{proof}

In  particular a   Calabi-Yau threefold with $\Q$-factorial terminal singularities satisfies the hypothesis of the Corollary. Note also that if $h^2(W,\cO_W)=0$  then $\ns(Y) \subseteq  H^2(Y, \Z)  $.

\begin{proposition} \label{Cartan-proposition} Let  $B$ be smooth, $W \to B$ be a Weierstrass model and $X \to B$  the birationally equivalent minimal model  with $\Q$-factorial terminal  Gorenstein singularities. Equivalently, let $X \to B$ be an elliptic  threefold with  Gorenstein $\Q$-factorial terminal singularities and $W\to B$ the associated minimal model. The pairing restricted to  $L$ and to $\Lambda$ gives the negative of the Cartan matrices of the algebras $\mathfrak g (\Sigma_j)$  as in Proposition \ref{TateAlg}.
\end{proposition}
$\Lambda$ serves as the coroot lattice of the Lie group $G$ associated with $\mathfrak{g}$, and $\Lambda\otimes
U(1)$ serves as the Cartan subgroup, as in \cite[Lemma 1.2]{GrassiMorrison03}. 
\begin{proof}
 The partial resolution constructed in Tate's algorithm  is isomorphic to $X$ over the general points of $\Sigma$. Recall that  the fibration $\pi: X \to B$ is equidimensional \cite{GrassiEqui}.  
 Let $D^l_j$ and $\ell_{j,l}$ be defined as in  Corollary \ref{def:g'}.
 The curves $\{\ell_{j,l}\}$ generate $L$ over $\Q$, $\{D^l_j\}$ generate $\Lambda$;  $<\ell_{j,k} , D^l_j>= \ell_{j,k} \cdot D^l_j$  gives the negative of the entries of a block of the Cartan matrix.\end{proof}

 \begin{remark}\label{remarkCartan}
 The Poincar\'e pairing between $\Lambda$ and $L$  gives the transpose Cartan matrix; recall that the  Cartan matrix is not symmetric if $\mathfrak g$ is not simply laced. Note  also that 
 the rows of  (a block in) the Cartan matrix  are the Dynkin coefficients of the roots, the weights of the adjoint representation.
 In fact, we will see in the following Section \ref{subsection_reps} that associated to $\Sigma_j$ is an
  ``unlocalized" representation, which is precisely given by the adjoint representation for the simply laced algebras. 
\end{remark}

\subsection{Representations of (gauge) algebras and  the codimension two strata}\label{subsection_reps}~

The support of the discriminant locus $\Sigma$ of  a genus one fibration  
 has  a stratified structure given by its singularities. In this paper we focus on the codimension one strata, given by the irreducible components of codimension one in $B$, and  the codimension two strata, given by the singular locus of $\Sigma$ and more generally by  the codimension two components of $\Sigma$ in $B$.  To simplify the statements we assume that $\dim B=2$, however the statements also hold for higher dimensions with appropriate modifications. 
If $\dim B=2$, we denote by $Q$ a singular point of $\Sigma$.

The ``unlocalized "and ``localized" representations in the physics language are associated to the different codimension of the strata of  $\Sigma$ and both occur with a certain multiplicity which depends on the dimension of the base $B$. 
We will  present methods for computing  the multiplicities if $B$ is a surface in the following sections. The methods  can be extended to the presence of singularities, as we prove under certain general assumptions in Section \ref{chitop}.

To simplify the statements we assume that  $B$ is smooth, $W \to B$ be a Weierstrass model and $X \to B$  the birationally equivalent minimal model  with $\Q$-factorial terminal  Gorenstein singularities. Equivalently, let $X \to B$ be an elliptic variety
with  Gorenstein $\Q$-factorial terminal singularities and $W\to B$ the associated minimal model. We can assume $X \to B$ to be equidimensional and $B$ smooth. Let $B_1$ be a section of the fibration.

\subsubsection{The unlocalized representations}[Codimension one strata]\label{unloc}~

\begin{lemma}\label{method-unloc}~
\begin{enumerate}
\item[(i)]The intersection product between $\ell_{j,k}$ and $D_j^l$ described  in Proposition \ref{Cartan-proposition} gives  the positive simple root vectors
$ \alpha^l_k =  - \ell_{j,k} \cdot   D_j^l$, which are weight vectors associated with the representation ${\rm adj}_{ {\mathfrak g}(\Sigma_j)}$.
\item [(ii)] If  $\mathfrak{g}(\Sigma_j)$ is not simply laced, there is another naturally associated representation  $\rho_0^{d-1}(\Sigma_j)$, described in Table \ref{tab:A}.
\end{enumerate}
\end{lemma} 

\begin{proof}
$X$ is smooth over the generic point of the codimension one strata of $\Sigma$;
see also Remark \ref{remarkCartan}. This proves (i).
If $\mathfrak{g}(\Sigma_j)$ is not simply laced,
let  ${\tilde {\mathfrak g}}(\Sigma_j)$  be the Lie algebra associated with the { generic} fiber of $\Sigma_j$; it is a cover of the Lie algebra ${\mathfrak g}(\Sigma_j)$.       
$\rho_0(\Sigma_j)$  is the representation 
determined through the ``branching rules" which decompose
${\rm adj}_{ {\tilde {\mathfrak g}}(\Sigma_j) } $ as $ {\rm adj}_{ {\mathfrak g}(\Sigma_j)} \oplus \rho_0^{d-1}(\Sigma_j) \,,$
where $d$ is the degree of the cover introduced in Lemma \ref{def:g'} \cite{GrassiMorrison03}.
\end{proof}
The Lemma motivates the following
\begin{method}[Unlocalized representations]~
To each irreducible component of the codimension one strata $\Sigma_j$ one associates an {\it unlocalized representation} of $\mathfrak{g}(\Sigma_j)$ as follows:\\
If $\mathfrak{g}(\Sigma_j)$ is simply laced, the unlocalized representation is ${\rm adj}_{ {\mathfrak g}(\Sigma_j)}$.\\
If  $\mathfrak{g}(\Sigma_j)$ is not simply laced,  the unlocalized representation associated with $\Sigma_j$ is 
${\rm adj}_{ {\mathfrak g}(\Sigma_j)} \oplus \rho_0(\Sigma_j) \,,$
where $\rho_0(\Sigma_j)$ is summarized in Table \ref{tab:A}.

 If $\g (\Sigma_j)$ is simply laced, the multiplicity of the representation is  the genus of $\Sigma_j$, $g(\Sigma_j)$; if  $\g (\Sigma_j)$ is not simply laced, the multiplicity 
is $g(\Sigma'_j)- g(\Sigma_j)$, with $\Sigma'_j$ as in Lemma \ref{def:g'}.

\end{method}


\subsubsection{The localized representations}[Codimension two strata]\label{sec_localized}

 Let $Q$ be a singular point of  the discriminant $\Sigma$.
 The physics  predicts that certain representations are associated to $Q$.  We
present two ways, Method \ref{Method1} and  Method \ref{Method-KV}, to compute the unlocalized representations, building on \cite{Witten:1996qb},\cite{Intriligator:1997pq,Aspinwall:2000kf} and then \cite{KatzVafaMatter} as elaborated further in \cite{GrassiMorrison11}. Our general results imply that the methods can be extended to the case of  $\mathbb Q$-factorial terminal singularities, as we will verify explicitly under certain genericity assumptions in Section \ref{chitop}.

 The underlying principles and some first computations were outlined in \cite{Witten:1996qb},\cite{
Intriligator:1997pq,Aspinwall:2000kf} and \cite{KatzVafaMatter}. Various refinements and verifications, on smooth fibrations, have been made in the physics literature since.
 In the case of Calabi-Yau varieties, we will also show that the representations are independent of  the choice of the minimal model.
 
 \smallskip

The constructions described below  will give the trivial representation if $Q$ is replaced either by a general point of $B$, or a general point in  $\Sigma$ (the codimension one strata).

We make the following Conjecture, which  has been verified under general conditions even in the presence of $\Q$-factorial terminal singularities, as we will prove in this paper (see Table \ref{tab:A}), and in various other examples \cite{ArrasGrassiWeigand}:

\begin{conjecture}\label{Conjecture1}
 For $Q$ a singular point of the discriminant $\Sigma$ as above,  denote by $X_Q$ the corresponding fiber in $X$. Let  $\ell^a_Q$  be the class of an  irreducible component of $X_Q$  in ${H_2(X, \Z) \cap [B_1]^{\perp}}$.
 \begin{enumerate}
\item The intersection numbers with the ruled divisors $D_j^l$,
 $${ \beta^l(\ell^a_Q) =   - {\ell}^a_Q \cdot D_j^l,  \quad l = 1, \ldots, {\rm rk}(\mathfrak{g}({\Sigma_j})) \,,}$$
form the entries of a weight vector of an irreducible representation $\rho_{Q,a}$ of ${\mathfrak g}(\Sigma_j)$.
\item All weight vectors $\beta_p^l(\rho_{Q,a})$, labeled by $p \in \{1, \ldots, {\rm dim}(\rho_{Q,a}) \}$, are obtained by 
$\beta_p^l(\rho_{Q,a})  = -  C_p(\rho_{Q,a}) \cdot D_j^l $
with
$$C_p(\rho_{Q,a})   = \ell^a_Q + \sum_{k=1}^{{\rm rk}(\mathfrak{g}({\Sigma_j}))}  n_p^k  \, \ell_{j,k} ,\ \ n_p^k \in \mathbb Z \,.$$
\item Some of the curve classes $[C_p(\rho_{Q,a})]$ are represented by effective curves, and the remaining ones by anti-effective ones.
 \end{enumerate}
\end{conjecture}
In the algebra-geometry dictionary, (2) states that  all the weights are obtained  by adding to  a weight vector $\beta^l(\ell^a_Q)$ the linear combinations of the positive simple roots $\alpha^l_k$ with suitable coefficients $n_p^k \in \mathbb Z$,
$ \beta_p^l(\rho_{Q,a}) =  \beta^l(\ell^a_Q)  + \sum_{k=1}^{{\rm rk}(\mathfrak{g}({\Sigma_j}))} n_p^k  \,  \alpha^l_k $.
{If $\ell^a_Q=\ell_{j,l}$ is the class of a ruling, then $\rho_{Q,a}={\rm adj}( {\mathfrak g}(\Sigma_j))$, in agreement with the first observation in the proof of Lemma \ref{method-unloc}.}

\smallskip 

\begin{definition}Let $C_p(\rho_{Q,a})   = \ell^a_Q + \sum_{k=1}^{{\rm rk}(\mathfrak{g}({\Sigma_j}))}  n_p^k  \, \ell_{j,k} , \ \ n_p^k \in \mathbb Z \, , $  as in Conjecture \ref{Conjecture1}.
 Let $M(\ell^a_Q)\stackrel{def}=\{C_p(\rho_{Q,a}) \}$  the collection of such curves and
$-M(\ell^a_Q)\stackrel{def}=\{-C_p(\rho_{Q,a})\}$.
\end{definition}

\begin{method}[Localized representations, via weight lattices from intersection theory] \label{Method1} With the notation above, we make the following assignments:
\begin{enumerate}
\item To each  irreducible fiber components $\ell^b_Q$, assign a  representation $\rho_{Q,b}$ as in Conjecture \ref{Conjecture1}. 
\item $\ell^a_Q \neq \ell^b_Q $ give independent representations if and only if  $M(\ell^b_Q) \neq \pm M(\ell^a_Q)$.
\item If $M(\ell^a_Q) = -M(\ell^a_Q)$ as a set, then  the assigned multiplicity to $\rho_{Q,a}$  is $\delta_a = \frac{1}{2}$, otherwise $\delta_a = 1$.
\item The full representation associated with $Q$, with respect to ${\mathfrak g}(\Sigma_j)$, is then
$$
\rho_Q = \sum_{\rho_{Q,a} \neq {\rm adj}_{\mathfrak{g}(\Sigma_j)}}  \delta_a \, \rho_{Q,a} \,,
$$
where the sum is over the independent representations different from  ${\rm adj}( {\mathfrak g}(\Sigma_j))$.
 \end{enumerate}
\end{method}

\begin{remark}
If $Q $ is at  the intersection of  two different components,  for example $\Sigma_i \cap \Sigma_j$,  and  ${\mathfrak g}(\Sigma_i) \neq \{e\}$ and ${\mathfrak g}(\Sigma_j) \neq \{e \}$, then
$Q$ gives rise to representations of ${\mathfrak g}(\Sigma_i) \oplus {\mathfrak g}(\Sigma_j)$. 
This must be taken into consideration in determining the final multiplicity of the representations at $Q$.
 \end{remark}

\begin{remark}
In this sense $L$ as in Definitions \ref{L1} and \ref{L2} together with the intersection pairing with $\Lambda$ defines the weight lattice of the total algebra $\mathfrak{g}$.
When $X$ is Calabi-Yau, two birationally equivalent resolutions $X$ and $X'$ of the same Weierstrass model $W$ give rise to the same representations $\rho_Q$ defined above, as studied in the physics literature e.g. in \cite{Intriligator:1997pq,Hayashi:2014kca,Esole:2014bka}.  We prove this in general in  Theorem \ref{Correps}. 
 \end{remark}

The novel aspect of this Section is also that the procedure outlined in Method \ref{Method1} continues to be applicable if $X$
has $\mathbb Q$-factorial terminal singularities. In this case, $X$ is the relative minimal model of 
$W$. 
However, special care must be taken in evaluating the intersection numbers determining the weights due to the presence of the singularities.
As we show  in  Case 3 in the proof of Theorem  \ref{Anomalycondition1} below, the singularity is associated to a (non)-trivial localized representation of the algebra.

\medskip

 Before describing the second method to determine the localized representations, we need the following 
 \begin{definition} \label{Defgz}
 Let $W \to Z$ be a  minimal Weierstrass model over a smooth surface $Z$. Let $z \in Z$ be a point and $C \subset Z$  be a general curve through $z$ in a (Euclidean) neighborhood of $z$. Consider the Weierstrass  surface  $W_{|C}$ restricted to $C$. Without loss of generality we assume also  that $W_{|C}$ defines  a minimal Weierstrass surface.
Then by  $\mathfrak{g} (z)$ we denote the ``gauge" algebra associated to $W_{|C}$ at the point $z$.
 \end{definition}
Note that the singularity of $W_{|C}$ in the fiber over $z$ is a rational double point, and $\mathfrak{g}(z)$ is the
simply laced Lie algebra with Dynkin diagram the dual graph of the curve of resolution.

\begin{method}[Localized representations, via ``Katz-Vafa's method"]\label{Method-KV}
Consider $W$, $B$ and $X$ as stated at the beginning of  Section \ref{subsection_reps}, and let $Q$ be a singular point of  the discriminant $\Sigma$.
\begin{enumerate}
\item Up to a change of parameter  $t = z^d$, there   is a family of disks  $C_{t}$ intersecting $\Sigma$ at $P_{t}$ with $P_0 = Q$ such that the singularities of $\pi^{-1}(C_t)$ admit a simultaneous resolution. 
\item Furthermore, there exists a space of versal deformations of $\pi^{-1}(C_0)$ which is simultaneously resolvable, and the parameter curve  $\{z\}$ of $C_{z^d}$ is a ramified cover of the parameter curve of the versal deformations
with ramification $c$ at $z=0$. Let  $b = d/c$ and locally $z^d = s^b$.
\item Then one can decompose
$${\rm adj}_{{\mathfrak{g}}(Q)} ={\rm \adj}_{{\mathfrak{g}}(Q_{s^b})}  \oplus \hat{\rho}_Q \oplus \bar{\hat{\rho}}_Q \oplus 1^{\oplus (\rk({\mathfrak{g}}_Q) -  \rk({\mathfrak{g}}(Q_{s^b}) )},
$$
with $\mathfrak{g}(Q)$ and $\mathfrak{g}(Q_{s^b})$ as in Definition \ref{Defgz}.
\item  { The representation $\rho_Q$ associated to $Q$ in $X$ is 
$$\rho_Q  = \frac{1}{b} {\rho}'_Q, $$
where ${\rho}'_Q$ follows from
decomposing $\hat{\rho}_Q$ into representations of $\mathfrak{g}$ 
as
\be \label{rhoQdecomp2}
 \hat{\rho}_Q =  \rho'_Q  \oplus  \rho_{\rm sing} \oplus    \oplus_i \frac{1}{2} \rho_0(\Sigma_i). 
 \ee
Here $\rho_{\rm sing}$ (if non-zero) is a singlet with respect to $\mathfrak{g}$ and the factors $\frac{1}{2} \rho_0(\Sigma_i)$  may appear in the decomposition  for those $\Sigma_i$ with $Q \in \Sigma_i$  only if $\mathfrak{g}(\Sigma_i)$ is non-simply laced and $Q$ is a ramification point for the associated monodromy. In this case $\rho_0(\Sigma_i)$ is as defined in Lemma \ref{method-unloc}.

}

\end{enumerate}
\end{method}

Note that (1) and (2)   are guaranteed by the Tyurina and Brieskorn-Grothendieck theorems \cite{BrieskornGrot,Slodowy}.

\begin{remark}
We show in the following section that Method \ref{Method-KV} continues to be applicable in the presence of $\mathbb Q$-factorial terminal singularities. 
\end{remark}

\begin{conjecture}Refinements of  Method \ref{Method1} and Method \ref{Method-KV} determine the same localized representations.
\end{conjecture}
The conjecture has been confirmed  for various classes of examples, but no general proof has been obtained.

\section{ (Gauge) algebras, representations, singularities and the topological Euler characteristic of Calabi-Yau threefolds}\label{chitop}

 Let   $X \to B$ be an elliptic Calabi-Yau threefold with $\Q$-factorial  terminal singularities and  Weierstrass model $W$, as in the previous sections.  
 
  Singular varieties are in fact unavoidable also in  the physics interpretation, even in the case of $\Q$-factorial Calabi-Yau threefolds with terminal singularities, when  there is a smoothing \cite{NamikawaSteenbrink}, but the smooth Calabi-Yau lies outside the loci of interest. This is the case of the Weierstrass models associated with the Jacobian of  general genus one fibrations without a section, which has $\Q$-factorial terminal singularities \cite{BraunMorrison}.
  For Calabi-Yau fourfolds it is known that even simple examples of isolated terminal singularities  of the type  $\C^4/ \Gamma$ cannot be smoothed \cite{SchlessingerInv1971,MorrisonStevens}.

 We find that while ``the gauge algebra" can be associated as in the smooth case, the  dictionary described in \cite{GrassiMorrison11, GrassiMorrison03, BMW}  between the ``anomaly constraints" in physics and the geometry of the Calabi-Yau must be modified  when the Calabi-Yau is singular.

 If $X=W$ is smooth, that is the gauge algebra is trivial,  $30 K_B^2 + \frac{1}{2} \chi_{top}(X) =0$ \cite[Theorem 2.2]{GrassiMorrison03}. More generally, we define the following invariant $\r$ and prove that, under general conditions, it contains information about (the dimensions of) certain representations of the associated gauge algebra.
 \begin{definition} \label{defr} Let
\bea
\r = 30 K_B^2 + \frac{1}{2}\left( \chi_{top}(X) +  \sum_P  m(P) \right),
\eea
where the sum is over the singular points $P$ of $X$ with Milnor number  $m(P)$. \end{definition}
By  Corollary \ref{mainbir}, $\r$ is independent of the choice of the particular minimal model $X$, the $\Q$-factorial terminal resolution of $W$. $\r$ is a topological invariant of $X$.

\subsection{Gauge algebra, general}\label{gags}

As in \cite{GrassiMorrison03}, we assume that the discriminant is of the form $\Sigma = \Sigma_1 \cup \Sigma_0$, where  $\Sigma_1$ is a smooth curve and $\Sigma_0$ denotes the locus where the general fiber is nodal   ($I_1$ fiber),   $\mathfrak{g}(\Sigma_0) = \emptyset$ and  $\mathfrak{g}(\Sigma_1)$ is  the associated Lie algebra as in Proposition \ref{TateAlg}. We also assume that the Weierstrass model is otherwise general   (``genericity assumption"). Let  $\lambda$ be the multiplicity  of  $\Sigma$ along $\Sigma_1$.
Our assumptions have the following implications, as summarized in  
 \begin{proposition}[Proposition 4.4 in \cite{GrassiMorrison03}]\label{link} 
\begin{enumerate}
\item \noindent $\Sigma_0 \cap \Sigma_1= \{ Q_1 ^1, \cdots, Q^{B_1} _1, Q^1 _2, \cdots, Q^{B_2} _2
\}.$

 The numbers $B_i$ are determined by the algebra $\mathfrak g$.
 \item (Equivalently:) $\Sigma_0 \cdot \Sigma _1 = (-12 K_B -\lambda \Sigma_1)
\cdot \Sigma_1= r_1 B_1 + r_2B_2,$ where the numbers $r_i$  and $\lambda$ are determined by the algebra $\mathfrak g$.
\item The local equation around    each point  $Q^\ell_i$ does not depend on $\ell$, but only on $i=1,2$;
 then without loss of generality we write $X_{Q_i}=\pi ^{-1} (Q_i^\ell)$.
\end{enumerate}
\end{proposition}

As in \cite{GrassiMorrison03} we make the following
\begin{definition} \label{def:chargedim}
Let $\rho$ be a representation of a Lie algebra
$\mathfrak{g}$,
with Cartan subalgebra $\mathfrak{h}$.
The {\em charged dimension of $\rho$} is
$(\dim \rho)_{ch}= \dim (\rho)- \dim (ker \rho|_{\mathfrak{h}})$.
\end{definition}
For example, if $\rho$ is the adjoint representation then
$$(\dim \adj) _{ch}= \dim \mathfrak{g}- \dim \mathfrak{h}= \dim G - \rk G.$$

\begin{thm} \label{Anomalycondition1} 
Let $X, W$ and $\Sigma$ be as above and $\r$ as in Definition \ref{defr}; let $Q \in \Sigma_1$ denote the singular points of $\Sigma$. 
Let $\rho_Q$ be the 
 associated localized representation obtained as in Section \ref{sec_localized}.
   $\rho_Q$ is given in Table \ref{tab:A}, a modified version of Table A in \cite{GrassiMorrison11}. 
Let $P$ denote the singular points of $X$ with Tyurina number $\tau(P)$. 
Then 
\begin{eqnarray} \label{requation}
 \r &=& (g-1) (\dim \adj) _{ch} + (g' -g)  (\dim \rho_0)_{ch}  + \sum_Q    (\dim \rho_Q)_{ch} +    \sum_P \tau(P) \,.
\end{eqnarray}
Here $g\stackrel{def}=g(\Sigma_1)$ and $g'\stackrel{def}=g(\Sigma'_1)$ denote the genus of the discriminant component $\Sigma_1$ and, respectively, of its finite branched cover $\Sigma_1'$ occurring in Lemma \ref{def:g'} and ${\rm adj} = {\rm adj}_{\mathfrak{g}(\Sigma_1)}$, $\rho_0 = \rho_0(\Sigma_1)$ are the unlocalized representations according to Lemma \ref{method-unloc}.
\end{thm}

\begin{table}[ht!]

\begin{center}
\scalebox{0.7}{
\begin{tabular}{|| c|c|c||c|c|c||c|c|c|c||} \hline
Number &Type & $ \mathfrak{g}$ &  $\rho_0$ &  $\rho_{Q^\ell _1}$ & $\rho_{Q^\ell_2}$ & $(\dim \adj) _{ch} $&${(\dim  {\rho_0})}_{ch}$&$\dim {( \rho_{Q^\ell_1})}_{ch}$&${\dim {( \rho_{Q^\ell_2})}_{ch}}$ \\ \hline
1 & $I_{1}$ & $\{e\} $    &    &-- &-- & $0$ &$0$&$0$  &  0 \\ \hline
2 &$I_{2}$ & $\su (2) $   &    &--&$\fund$ & $2$ &$0$&  $0$  & $2 $ \\ \hline
3 &$I_{3}$ & $\su (3) $      &  &-- &$\fund$  & $6$ &$0$&  $ 0$  & $3$   \\ \hline
4 &$I_{2k} $, $k\ge2$ & $\sp (k) $    &$\Lambda^2_0$&--&
    $\fund$ &$2k^2$  &$2k^2-2k$&$0$&$2k $    \\ \hline
5 &$I_{2k+1}$, $k\ge1$ & $\sp (k)$     &$\Lambda^2+2\times\fund$ &$\frac12\fund$  & $\fund$ &$2k^2$  &$2k^2+2k$&$k$ &  $2k$ \\ \hline
6 &$I_{n}$, $n\ge4$ & $\su (n) $ && $\Lambda^2$&  $\fund$ & $n^2 -n$ &$0$& $\frac12(n^2-n) $  & $n $ \\ \hline
7 &$II$ & $\{e\}$ &  &-- & & $0$ &$0$&  0  &     \\ \hline
8 &$III$ & $\su (2)$ &&   $2\times\fund$&  & $2$ &$0$&$4$  &  \\ \hline
9 &$IV$ & $\sp(1)$ & $\Lambda^2+2\times\fund$&$\frac12\fund$ & & $2$ &$4$ & $1$  &    \\ \hline
10 &$IV$ & $\su (3)$ && $3\times\fund$& & $6$ &$0$&  $9 $  &\\ \hline
11 &$I_0^*$ & $\mathfrak g_2$ &$\mathbf{7}$&-- &  &  $12$ &$6$&$0$ &    \\ \hline
12 &$I^*_{0}$ & $\spin (7) $&$\vect$&-- &$\spinrep$  &  $18$ &$6 $&$0$ & $8 $\\ \hline
13 &$I^*_{0} $ & $\spin (8) $ &&  $\vect$&$\spinrep_\pm$ &  $24$ &$0$&$8$ & $8$    \\ \hline
14 &$I^*_{1}$ & $\spin (9) $ &$\vect$&-- &$\spinrep$ &  $32$ &$8$&$0$  &$16$  \\ \hline
15 &$I^*_{1} $ & $\spin (10) $ && $\vect$ & $\spinrep_\pm$&  $40$ &$0$&$10$ & $16$     \\
  \hline
16 &$I^*_{2}$ & $\spin (11) $&$\vect$&-- &$\frac12\spinrep$  &  $50$ &$10$&$0$  &$16 $  \\
   \hline
17 &$I^*_{2} $ & $\spin (12) $ &&  $\vect$& $\frac12\spinrep_\pm $ &  $60$ &$0$& $12$& $16$\\ \hline
18 &$I^*_{n}$, $n\ge3$ &$\so (2n+7)$&$\vect$&-- &\nonmin &  $2(n{+}3)^2$ &$2n{+}6$&$0$  & \nonmin \\
   \hline
19 &$I^*_{n} $, $n\ge3$ & $\so (2n+8) $&&  $\vect$&\nonmin&  $2(n{+}3)(n{+}4)$ &$0$& $2n{+}8 $   & \nonmin    \\ \hline
20 &$IV ^*$ & $\mathfrak f_4$ &$\mathbf{26}$&--&  &  $48$ &$24$&$0$ &\\ \hline
21 & $IV ^*$ & $\mathfrak e_6$ &&  $\mathbf{27}$ &  & $72$&$0$& $27 $  &  \\ \hline
22 & $III ^*$ & $\mathfrak e_7$ & &$\frac12\mathbf{56}$ &  &  $126$ &$0$& $28 $    & \\ \hline
23 & $II ^*$ & $\mathfrak e_8$ & &\nonmin  & & $240$ &$0$&  \nonmin  &      \\ \hline
\end{tabular}
}
\end{center}
\medskip
\caption{ \small{The representations which occur under our ``generic''
hypotheses. 
The associated representation is independent of the particular resolution. Cases with
non-minimal Weierstrass model are denoted ``\nonmin''. 
${(\dim  \rho_i)_{ch}= {\boldmath \mathcal R}_i .}$}}\label{tab:A}
\end{table}

\begin{table}[ht]
{
\renewcommand{\arraystretch}{1.5}

\begin{center}
\scalebox{0.9}{
\begin{tabular}{||c|c|c|c|c|c|c|c|c||} \hline
Number & Type & $  \mathfrak{g} $   & $Q^\ell_1$ & $Q^\ell_2$ & $ \chi(X_{Q^\ell_1}) $ & $ \chi(X_{Q^\ell_2})$ & $\tau(P_1)$ & $\tau (P_2)$\\ \hline
1 & $I_{1}$ &  $\{e\}$ &   $II$   & $I_1$   \,  (\terminal)  & 2 ($II$) & 1 &0 &1\\ \hline
5 & $I_{2k+1}, k \ge1$ & $\sp (k)$ &  $I_{2k-2}^\ast$ &  $I_{2k+1}$ \,  (\terminal)   &   $k+2$ (br.)&  2k+1  &1 & 0 \\ \hline
7 & $II$ & $\{e\}$ & $III$  \,  (\terminal)  & & 2   & & 2 & \\  \hline
\end{tabular}}
\end{center}
}
\medskip
\caption{\small{The fiber types listed in column 4 and in column 5  correspond to the vanishing orders of  the Weierstrass model and not to the topology of the fiber  $X_{Q^\ell_i}$ of the minimal terminal $\Q$-factorial resolution.  
In the last two columns we list the Milnor-Tyurina numbers  at the points  with $\mathbb Q$-factorial terminal singularities (with no small resolution \terminal).
}}\label{table:small}
\end{table}


\begin{proof}[Proof of Theorem \ref{Anomalycondition1}]
If $X$ is smooth, $m(P) = \tau(P)= 0$, the Theorem has been proved in  \cite{GrassiMorrison03} 
 by deconstructing $\chi_{ top}(X)$ with the help of the Mayer-Vietoris sequence and explicitly comparing both sides of (\ref{requation}) for all possible 20 types of Weierstrass models subject to the stated assumptions. 
 In particular, the dimensions ${\boldmath \mathcal R}_i = (\dim  \rho_i)_{ch}$ include the multiplicities of the representations as given in Table \ref{tab:A} . 
  If $X$ is singular, in our general hypothesis the Milnor and Tyurina number are equal \cite{SaitoK1971}. There are three more cases to consider.
 These are listed as Models number $1$, $5$ and $7$ in  Tables \ref{tab:A}, \ref{table:small}  taken from \cite{GrassiMorrison03} with the information on the singular models completed. In all these cases, $\chi_{top}(X)$ is computed via deconstruction \cite{GrassiMorrison03} as summarized in equ. (A.11) of \cite{ArrasGrassiWeigand}.
 \\
{\bf Case 1 (Model 1 in Table \ref{tab:A}):}
The fiber over generic points of $\Sigma_1$ is of Kodaira Type $I_1$ and the gauge algebra is trivial.
In the fibers over the  points  $Q^\ell_2$ there are $\mathbb Q$-factorial terminal singularities with Milnor and Tyurina numbers $m(P_2) =\tau(P_2)= 1$ (Kleinian $A_2$ or conifold); note that these are not locally analytically $\mathbb Q$-factorial. 
Since the gauge group is trivial, no charged representations are  present. The  claim then follows. 
 \\
{\bf Case 2 (Model 7 in Table \ref{tab:A}):}
The fiber over generic points of $\Sigma_1$ is of Kodaira Type $II$ and  the gauge algebra is trivial.
In the fibers over the  points  $Q^\ell_1$ there are $\mathbb Q$-factorial terminal singularities (Kleinian $A_3$) with $m(P_1) =\tau(P_1) = 2$.  
There are again no charged representations, and the Theorem follows.\\
{\bf Case 3 (Model 5 in Table \ref{tab:A}):} 
The fiber over generic points of $\Sigma_1$ is locally of Kodaira Type $I_{2k+1}$ 
and associated gauge algebra $\mathfrak g= \mathfrak{sp}(k)$. 
In the fibers over the  points  $Q^\ell_2$ there are  $\mathbb Q$-factorial  terminal singularities with $m(P_2)= \tau(P_2) =1$; topologically the fiber at $Q^\ell_2$ is the same as the general fiber over $\Sigma_1$. 
We claim that the singularity induces a  representation associated to $Q^\ell_2$, and that it  is  $\rho_2 = \fund_{\mathfrak{sp}(k)}$. 
This representation can be understood by considering the double cover $\hat X$ of the fibration  with $I_{2k+1}$ fibers over generic points of $\Sigma_1$, with associated algebra $\mathfrak{su}(2k+2)$. In the language of Method \ref{Method-KV}, this double cover admits a simultaneous resolution. Hence $t = z^2$, i.e. $d=2$.
In the double cover, the  the fiber at $Q^\ell_2$ becomes $I_{2k+2}$.
The versal deformations of this singularity  in $\hat X$  are parametrized by a deformation parameter $s = t =z^2$, leading to the parameter $b=1$. 
Following  Method \ref{Method-KV} (and the ``branching rules") we decompose the adjoint of $\mathfrak{g}(Q_2^\ell) = \mathfrak{su}(2k+2)$ into a representation of $\mathfrak{g}((Q_2^\ell)_s) = \mathfrak{su}(2k+1)$,
\begin{eqnarray}
\mathfrak{su}(2k+2) &\rightarrow& \mathfrak{su}(2k+1) \oplus \mathfrak{u}(1) \\
\adj_{\mathfrak{su}(2k+2)} &\rightarrow& (\adj_{\mathfrak{su}(2k+1)})_0 +  (\adj_{\mathfrak{u}(1)})_0  +  (\fund_{\mathfrak{su}(2k+1)})_1 + \overline{(\fund_{\mathfrak{su}(2k+1)})_{1}}
\end{eqnarray}
and further decompose $\fund_{\mathfrak{su}(2k+1)}$ into a representation of $\mathfrak{sp}(k)$, 
\begin{eqnarray}
\fund_{\mathfrak{su}(2k+1)} |_{\mathfrak{sp}(k)} = 1 + \fund_{\mathfrak{sp}(k)}.
\end{eqnarray}
Note that this decomposition is consistent with the form of formula (\ref{rhoQdecomp2}) because $Q_2^\ell$ is not a branch point for monodromy.
The  $\mathfrak{sp}(k)$ charged part of the decomposition is the representation associated with $Q^\ell_2$, i.e. $\rho_2 = \fund_{\mathfrak{sp}(k)}$.
In addition $\rho_0 = {\bf \Lambda}^2 + 2 \fund$ \cite{GrassiMorrison03}.
The RHS of (\ref{requation}) evaluates to $(g-1) ({\rm dim}(\mathfrak{g}) - \rk (\mathfrak{g})) + (g'-g) {\mathcal R}_0 + B_1 {\mathcal R}_1 + B_2  {\mathcal R}_2 + B_2$, with ${\mathcal R}_i= (\dim  \rho_i)_{ch}$ as given in Table \ref{tab:A} and with $g'-g = \frac{1}{2} \Sigma_1 \cdot (\Sigma_1 - K_B)$, $g-1 = \frac{1}{2} \Sigma_1 \cdot (\Sigma_1 + K_B)$ \cite{GrassiMorrison03}. The claim follows. 
\end{proof}

\subsection{F-theory interpretation}\label{FtheoryInterpretation}  In the physics literature mostly smooth models have been considered; in particular the ``F-theory" interpretation of the correpondence between singularities and algebras is on manifolds. However,  $\Q$-factorial terminal  singularities occur naturally, for example in the Jacobian variety of a genus one fibration  \cite{BraunMorrison}, in certain fiber products of rational elliptic surfaces  \cite{Morrison:2016lix},
 as well as F-theory duals of generic non-geometric compactifications of the heterotic string as studied in \cite{LustTFects,Font:2017cya}. 

Theorem \ref{Anomalycondition1}  is consistent with the cancellation of gravitational anomalies in the six-dimen-sional effective theory obtained by compactification of ``F-theory" on $X$, 
\begin{equation} \label{HminusV}
H - V + 29 T  = 273, 
\end{equation}
even when $X$ has singularities. Here
\begin{equation*}
{\rm dim}(\mathfrak{g})=V, \ \ h^{1,1}(B) -1=T
\end{equation*}
 are the  ``number of vector multiplets" and the ``number of tensor multiplets", respectively.
  The ``number of hypermultiplets" $H$ splits into the number of  ``charged" and ``uncharged"  hypermultiplets,
\begin{equation*}
H = H_{ unch} + H_{ ch} ,  \end{equation*}
where $ H_{ unch}  $  is the dimension of  the complex deformations $+1$  and $H_{ch}$  is related to the dimension of the algebra and its representations, with multiplicities. 
Theorem \ref{Anomalycondition1} indicates that  $H_{ch}$ and $H_{unch}$ decompose in localized and unlocalized summands.  We make the following

\begin{definition}\label{Hch}~

\begin{enumerate}
\item[(i)] $H_{ch} =  H_{ch}^{unloc} +  H_{ch}^{loc}$ , with  \\
$H_{ch}^{unloc} =   g \, (\dim \adj_\mathfrak{g})_{ch}  +  (g' - g) (\dim \rho_0)_{ ch}$ and 
 $H_{ ch}^{loc} = \sum_Q    (\dim \rho_Q)_{ch} .$
\item[(ii)] $H_{unch} =  H_{unch}^{unloc} +  H_{unch}^{loc}$, with $\\
H_{ unch}  =1 + \cD(X)$ and  $H_{unch}^{loc} = \sum_P \tau(P) .$
\end{enumerate}
\end{definition}
The ``uncharged localized" hypermultiplets counted by $H_{ unch}^{loc} = \sum_P \tau(P)$ are the number of versal deformations of the singularities at $P$ in  Remark \ref{remarklocalized}. The splitting motivates Conjecture \ref{localtoglobaldef}.

With this definition, and in presence of singularities, Theorem \ref{Anomalycondition1} is equivalent to the gravitational anomaly cancellation condition (\ref{HminusV}), where one has to use that 
$K_B^2 = 10 - h^{1,1}(B)$ and, 
for the models satisfying the "genericity assumption" of Theorem \ref{Anomalycondition1}, $h^{1,1}(X) = 1 + h^{1,1}(B) + {\rm rk}(\mathfrak{g})$ as well as $m(P) = \tau(P)$ \cite{GrassiMorrison11}, \cite{ArrasGrassiWeigand}.

\subsection{Mordell-Weil group and several components}~

When the Mordell-Weil group ${\rm MW}(X)$ of the elliptic fibration has rank $r$,  the gauge algebra includes also an abelian part $\frak{u}(1)^{\oplus r}$, 
see Section \ref{agas}.   The codimension two strata of  $\Sigma_0$  will include additional points ${C}_r$ (over which the fiber is  Kodaira type $I_2$, $X$ locally smooth).
Then according to our assignment,  we expect an associated representation associated to the abelian part of the algebra  $\frak{u}(1)^{\oplus r}$.
In the physics framework, this  is  the contribution of  the extra "charged singlets" localized at $C_r$ to $H_{ch}$  \cite{Morrison:2012ei}.
If $X$ is otherwise general then (\ref{requation})  in Theorem \ref{Anomalycondition1} continues to hold, but on the RHS we add the additional contribution of $\sum_{C_r}$. The explicit counting of the points $Q$ furthermore changes accordingly.

\begin{conjecture} 
Let $X \to B$ be an elliptic Calabi-Yau threefold with $\Q$-factorial terminal singularities $\{P\}$, the relative minimal model of a Weirstrass model $W \to B$.\\
Assume    ${\rm rk}({\rm MW}(X)) = r$,     that the discriminant is of the form $\Sigma= \Sigma_0 \cup \Sigma_1 \cup \ldots \cup \Sigma_N $ with  the simple algebra $\mathfrak{g}_i$  associated to $\Sigma_i$ and that the Weierstrass model is otherwise general. 
Let 
\begin{equation*}
\r' \stackrel{def}= 30 K_B^2 + \frac{1}{2}\left( \chi_{top}(X) -  \sum_P  m(P) + 2 \sum_P \tau(P) \right) \,.
\end{equation*}
Then 
\begin{equation*}
\r' = \sum_i(g_i-1) (\dim \adj_i) _{ch} + (g_i' -g_i)  (\dim \rho_{0,i})_{ch}  + \sum_Q    (\dim \rho_Q)_{ch} +     \sum_{C_r} 1 + \sum_P \tau(P) \,,
\end{equation*}
where $g_i = g(\Sigma_i)$, $g'_i = g(\Sigma'_i)$, and 
$Q$, $C_r$ are the codimension two strata of $\Sigma$.

If $Q \in \Sigma_i \cap \Sigma_j$,  the associated representation $\rho_Q$ is a tensor product representation with respect to $\mathfrak{g}_i \oplus \mathfrak{g}_j$.


\end{conjecture}
If $m(P) =\tau(P)$ for all $P$, $\r=\r'$.

\subsection{ Birational Kodaira Classification and  elliptic fibration of higher dimensional varieties}
\begin{thm} \label{Correps} The algebra and the representations are birational invariants of the $\Q$-factorial terminal minimal model $X \to B$. 
\end{thm}
\begin{proof}   $\r$ is a birational invariant of  the minimal model of the fibration $X \to B$   (Corollary  \ref{mainbir}). 
The gauge algebra and the unlocalized representations are  also invariant by construction. 
The proof of Theorem \ref{Anomalycondition1} shows 
that the possible flops in the fibers over the points $Q$  are isomorphisms  in the neighborhood of the fibers over $Q'$  for  $\rho_Q\neq \rho_{Q'}$.
\end{proof} 
The dimensions of the representations, listed in last four columns of Table  \ref{tab:A} uniquely,  determine the Kodaira type of the general fibers over the codimension one strata of the discriminant, as well as the fibers over the codimension two strata, up to birational transformations of the relative minimal model of the fibration.
Based on Theorem \ref{Correps} we make the following
\begin{conjecture}\label{KodairaExt} The   Kodaira classification of  singular fibers on relatively minimal elliptic  surfaces, with section, extends to the class of birationally equivalent  relatively minimal  elliptic   threefolds.
The classification is obtained by associating to the stratified discriminant  locus   of the fibration $\Sigma$ the non abelian gauge algebras and their representations.
\end{conjecture}
  We speculate that the classifications can be suitably extended in higher dimension, for example, in the case of fourfolds with the addition of information on the Yukawa coupling.
 Recent work   \cite{Anderson:2018heq} suggests that  multiple fibers  in  a Calabi-Yau  of  type $_m I_0$ are associated to  discrete torsion.
 \begin{conjecture}
 The   Kodaira classification of  singular fibers on relatively minimal genus one  surfaces extends to the class of birationally equivalent  relatively minimal genus one fibered  varieties as in \ref{KodairaExt}. The multiple fibers of multiplicity $m$ in  a Calabi-Yau  are associated to discrete torsion $\Z /m \Z$.
 \end{conjecture}

%


\bibliographystyle{siam}
\bibliography{bibio}
\end{document}


\maketitle

Our goal here is to prove a result (presumably well-known to experts), giving a context where 
the ``Lefschetz $(1,1)$-theorem'' holds for certain singular complex projective varieties.  
\begin{theorem}\label{Lefschetz}
Let $X$ be a normal projective variety over the field ${\mathbb C}$ of complex numbers. Let 
$\pi:Y\to X$ be a resolution of singularities. Assume: $R^1\pi_*{\mathcal O}_Y=0$ (this condition 
is independent of the choice of resolution). Then
\begin{enumerate}
 \item the singular cohomology $H^2(X,{\mathbb Z})$ supports a {\em pure} Hodge structure
 \item the image of the Chern class map $c_1: Pic\,(X)\to H^2(X,{\mathbb Z})$ concides with
 the subgroup of $(1,1)$ classes, i.e. with the subgroup
 \[\{\alpha\in H^2(X,{\mathbb Z})\,|\, \alpha_{\C}\in H^2(X,\C)\mbox{ is of type $(1,1)$}\}.\]
\end{enumerate}
\end{theorem}
\begin{example}
Let $X$ be a normal projective variety over $\C$ with $\dim X\geq 3$, and only isolated Cohen-Macaulay 
singularities (eg., isolated complete intersection singularities). Then $X$ satisfies the hypotheses 
of Theorem~\ref{Lefschetz}. (The argument below is presumably also well known to experts.) 

Indeed, let $S=\{x_1,\ldots,x_r\}$ be the singular locus. If $\pi:Y\to X$ is a resolution, 
then $R^1\pi_*{\mathcal O}_Y$ is supported within the singular locus, and so is a direct sum 
of skyscraper sheaves supported at the points $x_i$. So it suffices to show that the stalk at each 
$x_i$ vanishes. 

Let $Z_i=Y\times_X{\rm Spec}\, {\mathcal O}_{X,x_i}$, so that the stalk of 
$R^1\pi_*{\mathcal O}_Y$ at $x_i$ is $H^1(Z_i,{\mathcal O}_{Z_i})$. Let $E_i=\pi^{-1}(x_i)$. 

Then we have that 
\[H^1_{E_i}(Z_i,{\mathcal O}_{Z_i})=0\]
from the Grauert-Riemenschneider theorem. 

If $U_i=Z_i\setminus E_i$, which is also isomorphic to the punctured specrum of ${\mathcal O}_{X,x_i}$, then we have an exact sequence 
in local cohomology 
\[\cdots \to H^1_{E_i}(Z_i,{\mathcal O}_{Z_i})\to H^1(Z_i,{\mathcal O}_{Z_i})\to H^1(U_i,{\mathcal O}_{U_i})\to\cdots\]

On the other hand, we also have a sequence in local cohomology for the structure sehaf of ${\rm Spec}\,{\mathcal O}_{X,x_i}$ and 
its punctured spectrum; this yields an isomorphism
\[H^1(U_i,{\mathcal O}_{U_i})\cong H^2_{\mathfrak M_{x_i}}({\mathcal O}_{X,x_i}).\]

Finally, since I assumed $X$ is Cohen-Macaulay, {\em of dimension $\geq 3$}, the local cohomologies satisfy
\[H^j_{\mathfrak M_{x_i}}({\mathcal O}_{X,x_i})=0\;\;\forall\;j<\dim X,\]
and in particular for $j=2$. 

Thus $H^1(Z_i,{\mathcal O}_{Z_i})$ vanishes as claimed. (In fact, one has that $R^j\pi_*{\mathcal O}_Y=0$ for $j<\dim X-1$, by a similar argument.)
\end{example}
Now we sketch the proof of Theorem~\ref{Lefschetz}.
\begin{proof}
We now work with the analytic topology. By GAGA, the condition $R^1\pi_*{\mathcal O}_Y=0$, for the analytic topology, is the same
as the similar condition in the Zariski topology. Note also that ${\rm Pic}\,(X)$ is identified with the isomorphism classes of 
analytic line bundles.   

From the exponential sheaf sequence on $Y$, we have a long exact sequence of sheaves on $X$
\[0\to \pi_*\Z_Y\to \pi_*{\mathcal O}_Y\to {\mathcal O}^*_Y\stackrel{\partial}{\rightarrow} R^1\pi_*\Z_Y\to R^1\pi_*{\mathcal O}_Y\cdots\]
Since $X$ is normal, the first three terms are just the exponential sequence on $X$, and in particular, $\partial=0$. Since we assumed 
$R^1\pi_*{\mathcal O}_Y=0$, we get also that $R^1\pi_*\Z_Y=0$. 

Now the Leray spectral sequence implies that 
\begin{enumerate}
 \item[(i)] $H^1(X,\Z)\to H^1(Y,\Z)$ and $H^1(X,{\mathcal O}_X)\to H^1(Y,{\mathcal O}_Y)$ are isomorphisms 
 (this implies that ${\rm Pic}^0(X)\to {\rm Pic}^0(Y)$ is an isomorphism)
\item[(ii)] $H^2(X,\Z)\to H^2(Y,\Z)$ and $H^2(X,{\mathcal O}_X)\to H^2(Y,{\mathcal O}_Y)$ are injective.
 \end{enumerate}

 From (ii), it follows that $H^2(X,\Z)$ supports a {\em pure} Hodge structure, which necessarily satisfies $F^1H^2(X,\C)=H^2(X,\C)\cap F^1H^2(Y,\C)$, 
 and further, using the exponential sheaf sequence, that 
 $$\displaylines{
 NS(X)=\ker\left(H^2(X,\Z)\to H^2(X,{\mathcal O}_X)\right)=\ker\left(H^2(X,\Z)\to H^2(Y,{\mathcal O}_Y)\right)=\cr
 \ker\left(H^2(X,\Z)\to\frac{H^2(Y,\C)}{F^1H^2(Y,\C)}\right)=\ker\left(H^2(X,\Z)\to\frac{H^2(X,\C)}{F^1H^2(X,\C)}\right).\cr}$$
 Thus, $X$ satisfies the ``Lefschetz $(1,1)$ theorem''.

\end{proof}